\documentclass[preprint]{imsart}

\RequirePackage[OT1]{fontenc}
\RequirePackage{amsthm,amsmath,amsfonts,amssymb}
\RequirePackage[numbers]{natbib}
\RequirePackage[colorlinks,citecolor=blue,urlcolor=blue]{hyperref}
\usepackage{multirow}
\usepackage{graphicx}
\usepackage{mathrsfs}
\usepackage{bm}
\usepackage{epsfig}
\usepackage{epstopdf}
\usepackage{epic,rotating,color}

\startlocaldefs
\newsavebox{\tablebox}
\newcommand{\prob}{\mathbb{P}}
\newcommand{\expn}{\mathbb{E}}

\newtheorem{thm}{Theorem}

\newtheorem{lem}[thm]{Lemma}
\newtheorem{remark}{Remark}

\newcommand{\U}{\mathcal{U}}
\newcommand{\V}{\mathcal{V}}
\endlocaldefs

\begin{document}

\begin{frontmatter}

\title{Block bootstrap optimality for density estimation with dependent data}
\runtitle{Block bootstrap for density estimation}

\begin{aug}
\author{\fnms{Todd A.} \snm{Kuffner}\thanksref{m1}\corref{}\ead[label=e1]{kuffner@wustl.edu}},
\author{\fnms{Stephen M.S.} \snm{Lee}\thanksref{m2}\ead[label=e2]{smslee@hku.hk}}
\and
\author{\fnms{G. Alastair} \snm{Young}\thanksref{m3}\ead[label=e3]{alastair.young@imperial.ac.uk}}

\runauthor{T.A. Kuffner et al.}

\affiliation{Washington University in St. Louis\thanksmark{m1}, The University of Hong Kong\thanksmark{m2} and Imperial College London\thanksmark{m3}}

\address[m1]{Department of Mathematics and Statistics \\ 
St. Louis, MO, 63130 \\
USA \\
\printead{e1}}

\address[m2]{Department of Statistics and Actuarial Science \\ 
Pokfulam Road \\
Hong Kong \\
\printead{e2}}

\address[m3]{Department of Mathematics \\ 
London SW7 2AZ \\
United Kingdom \\
\printead{e3}}
\end{aug}

\begin{abstract}
Accurate approximation of the sampling distribution of nonparametric kernel density estimators is crucial for many statistical inference problems. Since these estimators have complex asymptotic distributions, bootstrap methods are often used for this purpose. With i.i.d. observations, a large literature exists concerning optimal bootstrap methods which achieve the fastest possible convergence rate of the bootstrap estimator of the sampling distribution of the kernel density estimator. With dependent data, such an optimality theory is an important open problem. We establish a general theory of optimality of the block bootstrap for kernel density estimation under weak dependence assumptions which are satisfied by many important time series models. 
We propose a unified framework for a theoretical study of a rich class of bootstrap methods which include as special cases subsampling, K\"{unsch}'s moving block bootstrap, Hall's under-smoothing (UNS) as well as approaches incorporating no (NBC) or explicit bias correction (EBC). Moreover, we consider their accuracy under a broad spectrum of choices of the bandwidth $h$, which include as an important special case the MSE-optimal choice, as well as other under-smoothed choices. Under each choice of $h$, we derive the optimal tuning parameters and compare optimal performances between the main subclasses (EBC, NBC, UNS) of the bootstrap methods.

\end{abstract}

\begin{keyword}[class=MSC]
\kwd[Primary ]{62G07, 62G09}
\kwd[; secondary ]{62G20, 62M10, 62E20}
\end{keyword}

\begin{keyword}
\kwd{Block bootstrap}
\kwd{kernel density estimation}
\kwd{subsampling}
\kwd{optimality}
\kwd{weak dependence}
\kwd{strong mixing}
\end{keyword}

\end{frontmatter}

\section{Introduction}
\label{sec:intro}

The goal of this paper is to establish a theory of optimality, in terms of achieving the fastest possible convergence rate, for the family of block bootstrap estimators of the sampling distribution of the kernel density estimator based on weakly dependent data. Given a sequence of observations $X_{1}, \ldots, X_{n}$ from a stationary time-dependent process $\{X_{i} \in \mathbb{R}:i =0, \pm 1, \pm 2, \ldots\}$, we consider estimating the marginal density $f$ of the process at a fixed $x_0\in\mathbb{R}$ using a kernel density estimator,
\begin{equation}
    \label{eqn:kde}
    \hat{f}_{h}(x_0)=(nh)^{-1}\sum_{i=1}^{n}K\left((X_{i}-x_0)/h\right), 
\end{equation}
where $h>0$ is a bandwidth and $K(\cdot)$ is a kernel function to be defined later. To facilitate inference about $f$, it is then natural to study the sampling distribution of 
\[
(nh)^{1/2}\big(\hat{f}_{h}(x_{0})-f(x_{0})\big),
\]
which is typically asymptotically normal, under suitable conditions on the underlying process and bandwidth parameter. 

Normal approximation to this sampling distribution is often inaccurate, and depends in a complicated way on both the bandwidth and dependence structure. Over the past several decades, a vast literature has developed concerning how to improve the accuracy of this approximation for independent data, and how to achieve faster convergence to the limiting normal distribution. Popular approaches include bootstrapping, studentization, bias correction, and optimal bandwidth selection, where optimality in the context of bandwidth is understood to mean that choice of $h$ which minimises an estimate of the risk under squared error loss. However, with dependent data, the properties of bootstrap estimators of the sampling distribution of the kernel density estimator are not well-understood. 

When observations are dependent, block bootstrap procedures are the standard tools for resampling-based inference; for a review of their many uses with time series data, see \citet{Buhlmann:2002, HHK:2003, Politis:2003, PRW:1999,Lahiri:2003} and \citet{KreissPaparoditis:2011}. Block bootstrap procedures sample $b$ blocks of consecutive observations of length $\ell$ from the original sample of $n$ observations, and then paste them together to form a pseudo-time series, so that each bootstrap sample consists of $b\ell$ observations. The two most popular variants of the block bootstrap are the subsampling bootstrap proposed by \citet{PolitisRomano:1994a}, which prescribes $b=1$, and the moving block bootstrap (MBB) due to \citet{Kunsch:1989} and \citet{LiuSingh:1992}, which prescribes $b=\lfloor n/\ell \rfloor$, where $\lfloor a \rfloor$ is the largest integer less than or equal to $a$. The subsampling bootstrap exactly preserves the dependence structure of the original sample, and the MBB reproduces the original dependence structure asymptotically. 

We consider block bootstrap estimation of the sampling distribution of $(nh)^{1/2}[\hat{f}_{h}(x_{0})-f(x_{0})]$.
Our results apply to an entire family of block bootstrap procedures. We will allow both $b$, such that $1 \leq b \leq \lfloor n/\ell \rfloor$,  and $\ell$, such that $1 \leq \ell \leq n$, as well as tuning parameters in the form of bootstrap-level bandwidths, to be chosen optimally, and thus the subsampling bootstrap and MBB are special cases of our more general block bootstrap. The optimal choice of $(b,\ell)$ and the tuning bandwidths, i.e. the choice which yields the fastest convergence rate of the block bootstrap estimator, will typically correspond to a choice of $b$ which is intermediate between the subsampling and MBB choices, and for this reason we call our procedure a hybrid block bootstrap.

While there is an extensive catalog of theoretical results concerning optimality of subsampling and the MBB for smooth functionals with independent and dependent data, and for nonsmooth functionals with independent data, there are essentially no existing optimality results of these block bootstrap procedures for nonsmooth functionals with dependent data. Even basic consistency results for the block bootstrap estimator, which are weaker than optimality results, are scarce for kernel density estimation. 

One complicating factor in studying optimal bootstrap procedures in this setting is that kernel density estimators are nonsmooth functionals, which are not covered by general optimality theory for smooth functionals. This feature is particularly relevant because optimality, in the sense of achieving the fastest possible convergence rate for the bootstrap estimator of the sampling distribution, requires the study of higher-order asymptotic properties. For such nonsmooth functionals, the validity of higher-order asymptotic arguments must be established on a case-by-case basis, and thus, unlike the well-known `smooth function model' \citep{Hall:1992}, the optimality theory for different nonsmooth functionals cannot be established by a unified approach. 

A further point of departure from existing literature on block bootstrap is that optimality theory developed in this paper requires different tools than those afforded by empirical process theory, which is typically used to establish basic results such as consistency and asymptotic normality. The error incurred in stochastically approximating the statistical functional by an empirical process is generally of a much larger order than the error incurred in the bootstrap approximation of the sampling distribution of interest. Therefore, arguments which are based on bootstrapping the empirical process approximation of the statistical functional of interest cannot be used to establish the more precise convergence rates provided by our higher-order asymptotic arguments. 

Our study of the higher-order asymptotic properties of the block bootstrap for kernel density estimators resolves several important open questions. We derive the orders of accuracy for a large class of block bootstrap methods which include as special cases the moving block bootstrap, subsampling and undersmoothing.  To achieve this, we derive asymptotic expansions for high-order cumulants and related properties of the kernel density estimator and its bootstrap version under weak dependence assumptions, based on which novel Edgeworth-type expansions are derived for the true and bootstrap distribution functions.

We first determine the order of error in the (non-bootstrap) normal approximation to the sampling distribution of $(nh)^{1/2}[\hat{f}_{h}(x_{0})-f(x_{0})]$. Next, we show that the fastest convergence rate for the block bootstrap estimator of this sampling distribution is obtained by optimally selecting the block length, the number of blocks and the bootstrap-level bandwidths, which in general will correspond to a different procedure than the subsampling bootstrap and moving block bootstrap (MBB). Those block bootstraps prescribe fixed rules for the number of blocks, and existing theory for these procedures has only been concerned with choosing the block length. 
One exception is \citet{KLY:2018b}, where we solve the problem for sample quantiles. However, the theory for these two different nonsmooth functionals (kernel density estimators and sample quantiles) is strikingly different, as we shall discuss, and cannot be treated by a unified general theory.

The remainder of the paper is organized as follows. Background material is detailed in \S~\ref{sec:setting}. In \S~\ref{sec:backgrounds} we summarize key aspects of the block bootstrap and of kernel density estimators, and provide further information about previous work on different but relevant problems. The main theoretical results are presented in \S~\ref{sec:theory}. A simulation study in \S~\ref{sec:simulation} illustrates key facets of the theory for a stationary ARMA process, and we conclude with discussion in \S~\ref{sec:conclusion}. Proofs are contained in the Appendix.

\section{Problem Setting and background}
\label{sec:setting}

Let $\mathbb{Z} \equiv \{0, \pm 1, \pm 2, \ldots \}$ be the set of all integers. Define $\{X_{i}\}_{i \in \mathbb{Z}}$ to be a doubly-infinite sequence of random variables on the probability space $(\Omega, \mathcal{F}, P)$. It is assumed throughout that $\{X_{i}\}_{i \in \mathbb{Z}}$ is a strictly stationary process. The sequence $(X_{1}, \ldots, X_{n})$ denotes a sample of size $n$ from $\{X_{i}\}_{i \in \mathbb{Z}}$. 

\subsection{Strong Mixing}
\label{subsec:mixing}

We define dependence for the sequence of random variables $\{X_{i}\}_{i \in \mathbb{Z}}$ in terms of the mixing properties of $\sigma$-algebras generated by subsets of the sequence which are separated by a distance, in units of time, tending to infinity. For any two sub-$\sigma$-algebras of $\mathcal{F}$, say $\mathcal{F}_{1}$ and $\mathcal{F}_{2}$, the $\alpha$-mixing coefficient between $\mathcal{F}_{1}$ and $\mathcal{F}_{2}$ is defined to be \citep[Section 16.2.1]{AthreyaLahiri:2006}
\begin{equation}
\label{eqn:alpha}
\alpha(\mathcal{F}_{1}, \mathcal{F}_{2}) \equiv \sup_{A \in \mathcal{F}_{1}, B \in \mathcal{F}_{2}} \vert P(A \cap B)-P(A)P(B) \vert .
\end{equation}
Let $\mathcal{F}_{-\infty}^{k}$ be the $\sigma$-algebra generated by the random variables $X_{a}, X_{a+1}, \ldots, X_{k}$ as $a\rightarrow -\infty$, and $\mathcal{F}_{k}^{\infty}$ be the $\sigma$-algebra generated by $X_{k}, X_{k+1}, \ldots, X_{a}$ as $a\rightarrow \infty$. The $\alpha$-mixing coefficient of the sequence $\{X_{i}\}_{i \in \mathbb{Z}}$ is defined as
\[
\alpha(t) \equiv \sup_{k \in \mathbb{Z}} \alpha(\mathcal{F}_{-\infty}^{k}, \mathcal{F}_{k+t}^{\infty}),
\]
where $\alpha(\cdot, \cdot)$ is defined in \eqref{eqn:alpha}. If the $\alpha$-mixing coefficient decays to zero, 
\begin{equation}
\label{eqn:alphamix}
\lim_{t \rightarrow \infty} \alpha(t) = 0,
\end{equation}
then the process $\{X_{i}\}_{i \in \mathbb{Z}}$ is said to be strongly mixing. The sequence of random variables $\{X_{i}\}_{i \in \mathbb{Z}}$ is said to be weakly dependent if the process $\{X_{i} \}_{i \in \mathbb{Z}}$ is strongly mixing, i.e. if \eqref{eqn:alphamix} holds.

\subsection{The Block Bootstrap}

The MBB method \citep{Kunsch:1989} splits the original sample $X_{1}, \ldots, X_{n}$ into overlapping blocks of size $\ell$, $B_{i}=(X_{i}, \ldots, X_{i+\ell-1})$, together constituting a set $\{B_{1}, \ldots, B_{n-\ell+1}\}$. Let $B_{1}^{*}, \ldots, B_{b}^{*}$ be a random sample drawn with replacement from the original blocks, where $b=\lfloor n/\ell \rfloor$ is the number of blocks that will be pasted together to form a pseudo-time series. That $B_{1}^{*}, \ldots, B_{b}^{*}$ is a random sample from $\{B_{1}, \ldots, B_{n-\ell+1}\}$ means that the sampled blocks are independently and identically distributed according to a discrete uniform distribution on $\{B_{1}, \ldots, B_{n-\ell+1}\}$. The observations in the $i$th resampled block, $B_{i}^{*}$, are $X_{(i-1)\ell+1}^{*}, \ldots, X_{i\ell}^{*}$, for $i \leq 1 \leq b$. Then the MBB sample is the concatenation of the resampled blocks, written as 
\[
\underbrace{X_{1}^{*}, \ldots, X_{\ell}^{*}}_{B_{1}^{*},}, \underbrace{X_{\ell+1}^{*}, \ldots, X_{2\ell}^{*}}_{B_{2}^{*},}, \underbrace{X_{2\ell+1}^{*}, \ldots, X_{(b-1)\ell}^{*}}_{B_{3}^{*},\; \ldots\;,\; B_{b-1}^{*},}, \underbrace{X_{(b-1)\ell+1}^{*}, \ldots, X_{b\ell}^{*}}_{B_{b}^{*}}.
\]
Note that this way of constructing the pseudo-time series will reproduce the original dependence structure \textit{asymptotically}.

The subsampling bootstrap \citep{PolitisRomano:1994a}, and specifically the overlapping blocks version relevant to the present setting, first splits the original sample into precisely the same overlapping blocks as the MBB, each of length $\ell$. However, the subsampling bootstrap draws only a single block. A nice property of this procedure is that the original dependence structure in the sample is exactly retained in the single subsample. By contrast, the pseudo-time series constructed by the MBB only reproduces the original dependence structure asymptotically. The subsampling bootstrap is a special case of the MBB in which $b=1$.

Other block bootstraps include the circular block bootstrap \citep{PolitisRomano:1992}, subsampling bootstrap \citep{PolitisRomano:1994a}, stationary bootstrap \citep{PolitisRomano:1994b},  matched block bootstrap \citep{CDHHK:1998} and tapered block bootstrap \citep{PaparoditisPolitis:2001}. These various block bootstraps have been motivated by particular deficiencies in the original MBB, such as to reduce bias, improve variance estimation, or make the pseudo-time series closer to being stationary. In general, the performance of block bootstrap methods for any estimation or inference problem is influenced by the underlying dependence structure, the choice of the block length $\ell$, the choice of the number of blocks $b$, and of course the properties of the statistical functional being used to learn about the population quantity of interest. The validity of each block bootstrap method, i.e. consistency of the resulting estimator, must be studied on a case-by-case basis. We comment that convolved subsampling \citep{TPN:2017} may be viewed as a special case of our hybrid scheme, as the kernel density estimator has exactly a mean-like structure. However, convolved subsampling, based on any block length and any number of blocks, cannot achieve the optimal error rate.

Previous work on optimal block bootstrap methods for smooth functionals has emphasized the importance of choosing a block length $\ell$ which minimizes mean-squared error (MSE), and in all of the literature, the choice of the number of blocks $b$ was not part of the minimization problem. Instead, the number of blocks $b$ was motivated by other considerations. For example, the subsampling bootstrap prescribes $b=1$ because this exactly preserves the dependence structure in the original data, while the MBB prescribes $b=\lfloor n/\ell \rfloor$, which ensures that each bootstrap sample is of approximately the same size as the original sample, and that $b \rightarrow \infty$ as $n \rightarrow \infty$.

Subsampling and block bootstrap can be closely related to one another, at least in many standard smooth function applications. For the case of estimating sample means, the block bootstrap distribution is in fact a convolution of the subsampling bootstrap distribution \citep{TPN:2017}. This is not true in general, however. For example, it is not true for the sample quantile -- another non-smooth functional -- as pointed out by \citet{KLY:2018b}. For non-smooth functionals, one cannot hope to have a general theory of optimality. The present paper examines an important particular case.

\subsection{Nonparametric Bootstrap for Kernel Estimation}

Several other nonparametric bootstrap methods, distinct from block bootstrap, have been established as consistent for kernel estimators of various population quantities. In particular, for the specific problem of bootstrapping a kernel estimator of the conditional mean function, \citet{PaparoditisPolitis:2000} propose a local bootstrap approach which is distinct from the block bootstrap methods discussed herein. In the context of nonparametric trend estimation, \citet{KreissPaparoditis:2012} propose a hybrid wild bootstrap procedure. This involves using a frequency domain bootstrap, inverting the bootstrapped discrete Fourier transform, and thereby obtaining a time-domain pseudo-time series which is then used for kernel estimation of the trend function \citep[\S 3.2.2]{KreissPaparoditis:2012}. The frequency domain bootstrap (FDP) is a transformation-based bootstrap \citep[\S 2.9 and also Ch. 9]{Lahiri:2003} which attempts to transform the original data to be independent, or to at least weaken the dependence structure. It takes advantage of the fact that the Fourier transforms of weakly dependent observations are asymptotically independent \citep{Lahiri:2003b}.

\section{Bootstrapping Kernel Density Estimators}
\label{sec:backgrounds}

\subsection{Kernel Density Estimation Under Dependence}

Asymptotic properties of kernel estimators for time series models satisfying strong mixing conditions were studied by \citet{Robinson:1983} and \citet{MasryTjostheim:1995}. \citet{Robinson:1983} justified the use of kernel density estimators with weakly dependent data by showing that, under strong mixing conditions, there is first-order equivalence of the asymptotic distributions of the kernel density estimators from such strong mixing sequences to the analogous limiting distributions for such estimators arising from independent data. This phenomenon was called `whitening by windowing' by \citet{Hart:1996}.

\subsection{Explicit Bias Correction vs Undersmoothing}

Confidence intervals for density functions are easily derived from bootstrap approximation to the sampling distribution of kernel density estimators. Coverage properties of these intervals can depend crucially on bias correction, bandwidth selection, and the order of the kernels used. Conventionally, one optimizes the kernel density estimator by choosing that bandwidth which minimizes the asymptotic mean-square error (MSE) of the point estimator. The MSE-optimal bandwidth is too large, resulting in a non-negligible scaled bias of the kernel density estimator. \citet{Hall:1992} showed that undersmoothing leads to more accurate confidence intervals compared to explicit bias correction for kernel density estimators, and \citet{Neumann:1997} found similar results for kernel-based nonparametric regression estimators. Recently, \citet{CCF:2018} showed that this conventional wisdom may be incorrect. In particular, they showed that with independent data, explicit bias correction together with appropriate studentization can outperform undersmoothing in terms of coverage accuracy. The main issue in the original arguments of \citet{Hall:1992} which led to the preference for undersmoothing is that his bias-corrrected kernel density estimators were scaled by the asymptotic variance of the kernel density estimator alone, which does not take into account the variance of the bias estimator and therefore can be a poor approximation in finite samples. After correctly studentizing, coverage accuracy after explicit bias correction outperforms undersmoothing, and in fact the resulting coverage properties are somewhat robust to bandwidth specification. That is, the good coverage properties hold for a range of bandwidth values, rendering the tricky optimal bandwidth choice somewhat less essential. 

In the dependent data setting, there is also typically an undersmoothing assumption used to avoid explicit bias correction, and also to prove consistency of the bootstrap estimator of the sampling distribution for kernel-based estimators; c.f. \citet{NeumannKreiss:1998}, \citet{PaparoditisPolitis:2000}, \citet{FKM:2002}, \citet[p. 125; they have a typo and call it oversmoothing]{ParrellaVitale:2007}.
In the specialized setting of locally stationary processes, \citet{DPP:2013} study a local block bootstrap procedure for a nonparametric kernel estimate of a deterministic trend in the presence of $\alpha$-mixing errors. They prove consistency, and also consider optimal bandwidth for achieving the fastest convergence rate. Theorem 4 of \citet{DPP:2013} emphasizes that undersmoothing is necessary if one wants to avoid explicit bias correction. The local block bootstrap was originally proposed by \citet{PaparoditisPolitis:2002} and consistency for the sample mean was shown in \citet{DPP:2003}.

\subsection{Our Optimality Results}

We consider a general class of block bootstrap methods and derive their optimal settings within three subclasses corresponding to, respectively, explicit bias correction (EBC), no bias correction (NBC) and undersmoothing (UNS). 
Our results show that the fastest convergence rate for the block bootstrap estimate of the sampling distribution of the kernel density estimator is achieved by simultaneously choosing the pair $(b,\ell)$ and the bootstrap-level bandwidths. Thus, in addition to addressing the open problem of optimal block bootstrap methods for kernel density estimation, our results suggest that there are meaningful practical advantages to using a hybrid block bootstrap which uses a number of blocks in-between $b=1$ (subsampling) and $b=\lfloor n/\ell \rfloor$ (MBB). Our results carry important implications for scalability and computational efficiency of bootstrap methods with dependent data.

Our theoretical results distinguish between strong mixing processes with mixing coefficients decaying at, respectively, polynomial and exponential rates. Optimality theory is presented for both polynomial and exponential mixing. We first establish theorems concerning asymptotic normality of kernel density estimators under each mixing regime, as well as asymptotic normality for the corresponding block bootstrap kernel density estimators. The presence of bandwidth creates a number of interesting scenarios. The optimal number of blocks, and the bandwidths employed at the bootstrap stage, are intricately related to each other. 

\section{Theoretical Results}
\label{sec:theory}

\subsection{Distribution of kernel density estimator}

Assume that $(X_1,\ldots,X_n)$ is a sample of a stationary
strong mixing process with mixing
coefficient $\alpha(t)$ satisfying, for some $\beta>2$,
\[ \alpha(t)=O(t^{-\beta}),\qquad t\rightarrow\infty. \]
Denote by $f$ the density function of $X_1$ and by $f_{i_1\cdots i_d}$ the joint density function
of $\big(X_0,X_{i_1},\ldots,X_{i_d}\big)$ for any $i_d>\cdots>i_1>0$.  Assume that
\[ \underset{x}\sup|f(x)|<\infty\qquad\text{and}\qquad
\underset{x_0,x_1,\ldots,x_d}\sup\big|f_{i_1\cdots i_d}(x_0,x_1,\ldots,x_d)\big|<\infty,\]
for  any $i_d>\cdots>i_1>0$ and any $d=1,2,\ldots\:$.

Let $K$ be a univariate, non-negative, symmetric and bounded kernel function satisfying
\[ \int_{-\infty}^\infty K(u)\,du=1,\quad\mu_2\triangleq\int_{-\infty}^\infty u^2K(u)\,du<\infty,\quad\nu_2\triangleq
\int_{-\infty}^\infty K(u)^2du<\infty.\]

Let $x_0\in\mathbb{R}$ be fixed, at which  $f$ is assumed to be positive and four times continuously differentiable. Define, for a positive bandwidth $h>0$, the kernel density estimator of $f(x_0)$  to be
\[ \hat{f}_h(x_0)=(nh)^{-1}\sum_{i=1}^nK\big((X_i-x_0)/h\big). \]
Define, for $\beta>2$ and $d\ge 3$,
\begin{gather*}
\gamma_0(\beta,d)=1 - \pmb{1}\{\beta>d-1\}\left(\dfrac{\beta-1}{\beta+d-2}\right)
\left\{\dfrac{(\beta-1)(\beta+d-2)}{ \beta^2 + (d-3)\beta + (d-2)^2}\right\}^{\frac{(\beta-1)}{(\beta-d+1)}},\\
g_0(\beta)=\inf\left\{d\left(\dfrac{\beta}{\beta-1}-\gamma_0(\beta,d)\right):d\ge 3\right\}.
\end{gather*}
A plot of $g_0(\beta)$ against $\beta$ on the log-scale is shown below.
\begin{center}
\includegraphics[scale=0.35]{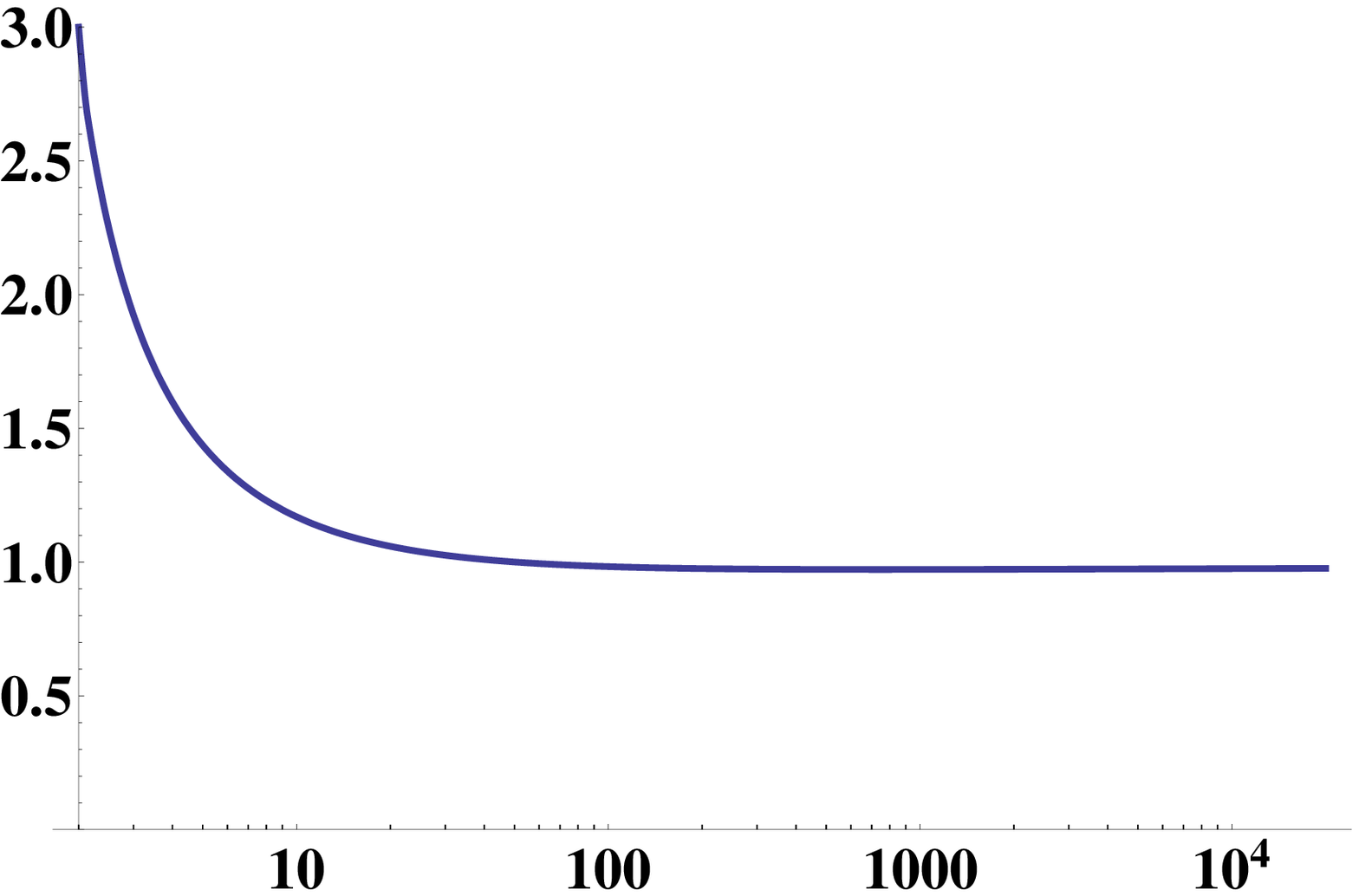}
\end{center}

\begin{thm}
\label{thm:nonbt}
Suppose that $\alpha(t)=O(t^{-\beta})$ as $t\rightarrow\infty$, for some $\beta>2$, and that
$nh^3\rightarrow\infty$ and $h=O(n^{-1/5})$.
Then, for any fixed $y\in\mathbb{R}$,
\begin{multline*}
\prob\Big((nh)^{1/2}\big\{\hat{f}_h(x_0)-f(x_0)\big\}\le y\Big)=
\Phi\left(\dfrac{y-n^{1/2}h^{5/2}f''(x_0)\mu_2/2}{\sqrt{f(x_0)\nu_2}}\right)\\
+\,O\left(n^{-(\beta-1)/(4\beta+2)}h^{-(\beta+1)/(4\beta+2)}+h^{(\beta-2)/\beta}
+n^{-1/2+\delta}h^{-(3\beta+3)/(2\beta-2)+g_0(\beta)}\right),
\end{multline*}
for any arbitrarily small $\delta>0$.
\end{thm}
It follows from Theorem~\ref{thm:nonbt} that the distribution of  $(nh)^{1/2}\big\{\hat{f}_h(x_0)-f(x_0)\big\}$
deviates from normality by a difference with minimum order
\begin{equation}
O\left(n^{-(\beta-1)(\beta-2)/(5\beta^2-5\beta-4)}\right),
\label{min.order}
\end{equation}
attained at $h\propto n^{-\beta(\beta-1)/(5\beta^2-5\beta-4)}$.

\begin{thm}
\label{thm:nonbt.expo}
Suppose that $\alpha(t)=O(e^{-Ct})$ as $t\rightarrow\infty$, for some $C>0$, and that
$nh\rightarrow\infty$ and $h=O(n^{-1/5})$.
Then, for any fixed $y\in\mathbb{R}$,
\begin{eqnarray*}
\lefteqn{
\prob\Big((nh)^{1/2}\big\{\hat{f}_h(x_0)-f(x_0)\big\}\le y\Big)}\\&=&
\Phi\left(\dfrac{y-n^{1/2}h^{5/2}f''(x_0)\mu_2/2}{\sqrt{f(x_0)\nu_2}}\right)+O\big\{h\log n+(nh)^{-1/2}\big\}.
\end{eqnarray*}
\end{thm}
It follows from Theorem~\ref{thm:nonbt.expo} that the distribution of  $(nh)^{1/2}\big\{\hat{f}_h(x_0)-f(x_0)\big\}$
deviates from normality by a difference with minimum order $O\big\{n^{-1/3}(\log n)^{1/3}\big\}$,
attained at $h\propto n^{-1/3}(\log n)^{-2/3}$.

\subsection{Block bootstrap estimation}
We consider a general class of moving block bootstrap estimators of the distribution of
\[T_h\triangleq
(nh)^{1/2}\big\{\hat{f}_h(x_0)-f(x_0)\big\}. \]

Define, for $k>0$, $\ell\in\{1,2,\ldots,n\}$ and $b\in\{1,2,\ldots\}$,
$J_1,\ldots,J_b$ to be independent random indices uniformly drawn from the
set $\{1,\ldots,n-\ell+1\}$,
\begin{gather*}
U_{i,k,\ell}=(\ell k)^{-1}\sum_{t=i}^{i+\ell-1}K\big((X_t-x_0)/k\big),\;\;
i=1,\ldots,n-\ell+1,\\
U^*_{i,k,\ell}=U_{J_i,k,\ell},\;\;
i=1,\ldots,b,\qquad\text{and}\qquad
\hat{f}^*_{b,\ell,k}=b^{-1}\sum_{i=1}^bU^*_{i,k,\ell}.
\end{gather*}
Note that $\expn\big[\hat{f}^*_{b,\ell,k}\big|X_1,\ldots,X_n\big]=(n-\ell+1)^{-1}\sum_{i=1}^{n-\ell+1}U_{i,k,\ell}$.

Define, for $\tau,k_1,k_2,k_3>0$,
\begin{multline*}
T^*_{b,\ell,\tau,k_1,k_2,k_3}=(b\ell k_1)^{1/2}\left\{\hat{f}^*_{b,\ell,k_1}-\expn\big[\hat{f}^*_{b,\ell,k_1}\big|X_1,\ldots,X_n\big]\right\}\\
+\tau\left\{\expn\big[\hat{f}^*_{b,\ell,k_2}\big|X_1,\ldots,X_n\big]-\hat{f}_{k_3}(x_0)\right\},
\end{multline*}
in which the second term $\tau\big\{\expn\big[\hat{f}^*_{b,\ell,k_2}\big|X_1,\ldots,X_n\big]-\hat{f}_{k_3}(x_0)\big\}$
can be regarded as a bias correction factor.

Define, for $\beta>2$,
\begin{align*}
g_1(\beta)&=(1/3)\inf\left\{d\big(2-\gamma_0(\beta,d)\big):d\ge 3\right\}-1,\\
g_2(\beta)&=(1/3)\inf\left\{d\big(2-\max\{2\gamma_0(\beta,2d)-1,0\}\big):d\ge 3\right\}-1.
\end{align*}
Plots of $g_1(\beta)$ (blue) and $g_2(\beta)$ (red) against $\beta$ on the log-scale are shown below.
\begin{center}
\includegraphics[scale=0.35]{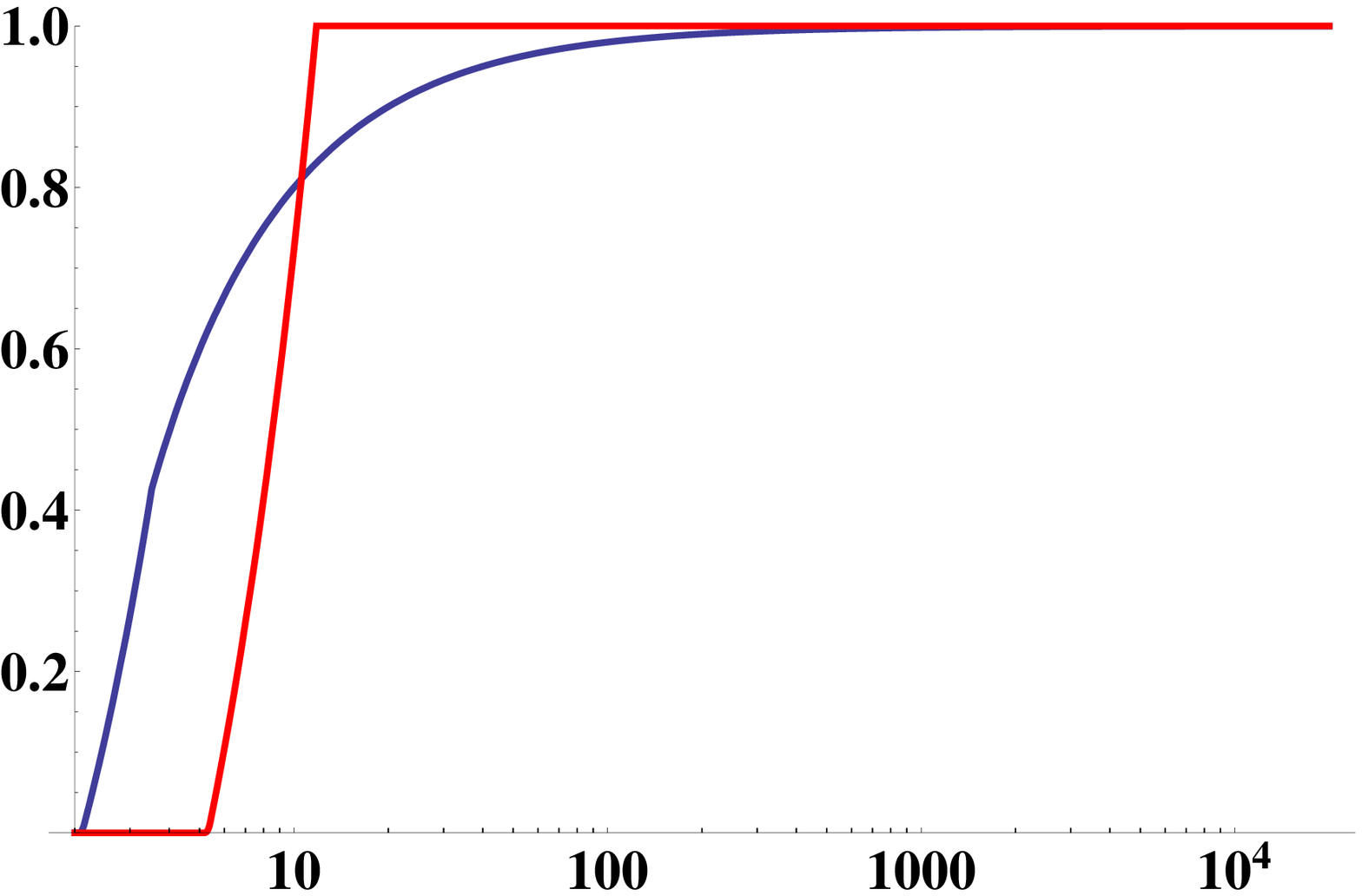}
\end{center}

\begin{thm}
\label{thm:boot}
Assume that $\alpha(t)=O(t^{-\beta})$ as $t\rightarrow\infty$, for some $\beta>2$.
Suppose that $b\ell=O(n)$, $\ell=O(bk_1)$ and
\[n^{-1}\ell+(\ell k_1)^{-1}+k_1+(nk_2)^{-1}+k_2+(nk_3)^{-1}+k_3=O(n^{-\epsilon}),\]
for some arbitrarily small $\epsilon>0$.
Let $y\in\mathbb{R}$ be fixed and $\delta>0$ be any arbitrarily small constant. Then
\begin{multline*}
\prob\big(T^*_{b,\ell,\tau,k_1,k_2,k_3}\le y\big|X_1,\ldots,X_n\big)=
\Phi\left(\dfrac{y-\tau(k_2^2-k_3^2)f''(x_0)\mu_2/2}{\sqrt{f(x_0)\nu_2}}\right)\\
+\,O_p\Big\{
k_1^{(\beta-2)/\beta}+b^{-1/2}\ell^{-1/2+\delta}k_1^{-3/2+3g_1(\beta)}+n^{-1/2+\delta}b^{-1/2}\ell^{1/2}k_1^{-3(1-g_2(\beta))/2}\\
+\,\tau\big[k_2^4+k_3^4+(nk_2)^{-1/2}+(nk_3)^{-1/2}\big]\Big\}.
\end{multline*}
\end{thm}
\begin{thm}
\label{thm:boot.expo}
Assume that $\alpha(t)=O(e^{-Ct})$ as $t\rightarrow\infty$, for some $C>0$.
Suppose that $b\ell=O(n)$ and
\[n^{-1}\ell+(\ell k_1)^{-1}+k_1+(nk_2)^{-1}+k_2+(nk_3)^{-1}+k_3=o(1).\]
Let $y\in\mathbb{R}$ be fixed. Then
\begin{multline*}
\prob\big(T^*_{b,\ell,\tau,k_1,k_2,k_3}\le y\big|X_1,\ldots,X_n\big)=
\Phi\left(\dfrac{y-\tau(k_2^2-k_3^2)f''(x_0)\mu_2/2}{\sqrt{f(x_0)\nu_2}}\right)\\
+\,O_p\Big\{k_1\log\ell+n^{-1/2}\ell^{1/2}+(b\ell k_1)^{-1/2}+\tau\big[k_2^4+k_3^4+(nk_2)^{-1/2}+(nk_3)^{-1/2}\big]\Big\}.
\end{multline*}
\end{thm}
It is clear by comparing the leading terms of the expansions established by
Theorems~\ref{thm:nonbt} and \ref{thm:boot} that we should set
$\tau(k_2^2-k_3^2)=(nh^5)^{1/2}$ to estimate consistently
the leading bias term $n^{1/2}h^{5/2}f''(x_0)\mu_2/2$. Minimising the error term
$\tau\big[k_2^4+k_3^4+(nk_2)^{-1/2}+(nk_3)^{-1/2}\big]$ subject to the latter condition suggests the following
settings for $(\tau,k_2,k_3)$:
\begin{equation}
k_2\propto n^{-1/9},\qquad k_3=c_0k_2\qquad\text{and}\qquad \tau=(1-c_0^2)^{-1}n^{1/2}h^{5/2}k_2^{-2},
\label{bias.cond}
\end{equation}
for some fixed $c_0\in(0,1)$.
Write, for brevity, $\hat{T}^*_{b,\ell,k_1}=T^*_{b,\ell,\tau,k_1,k_2,k_3}$, with
$\tau,k_2,k_3$ satisfying (\ref{bias.cond}). This gives rise to the EBC subclass of block bootstrap estimators which incorporate an explicit bias correction factor.

Similarly, in the absence of explicit bias correction such that $k_1=k_2=k$ and $\tau=(b\ell k)^{1/2}$,
we should set $k$ and $k_3$ to satisfy,
for some fixed $c_0\in(0,1)$,
\begin{equation}
k=(1-c_0^2)^{-2/5}n^{1/5}(b\ell)^{-1/5}h\qquad\text{and}\qquad k_3=c_0k.
\label{bias.cond.nocorr}
\end{equation}
Write, for brevity, $\tilde{T}^*_{b,\ell}=T^*_{b,\ell,(b\ell k)^{1/2},k,k,k_3}$, with
$k,k_3$ satisfying (\ref{bias.cond.nocorr}). This gives rise to the NBC subclass of estimators which involve no explicit bias correction.

Setting $\tau=0$ and $k_1=k$ in $T^*_{b,\ell,\tau,k_1,k_2,k_3}$ reduces to the special case
\[\check{T}^*_{b,\ell,k}=(b\ell k)^{1/2}\big\{\hat{f}^*_{b,\ell,k}-\expn[\hat{f}^*_{b,\ell,k}|X_1,\ldots,X_n]\big\},\]
which is applicable only if the bias is asymptotically negligible, or more precisely, if $h=o(n^{-1/5})$. This gives rise to the UNS subclass of undersmoothed block bootstrap estimators. Recall that if we set $b=n$, $\ell=1$
and $k=h$, the method corresponds to undersmoothing studied by \citet{Hall:1992} in the i.i.d.\ case.

\subsection{Choice of $(b,\ell,k_1)$ for EBC}
Note that under \eqref{bias.cond}, we have 
\[ \tau\big[k_2^4+k_3^4+(nk_2)^{-1/2}+(nk_3)^{-1/2}\big]=O\big(n^{5/18}h^{5/2}\big)=o\big(\min\{n^{-1/5},h\}\big), \]
which is smaller than the minimum order (\ref{min.order}). It then follows 
from Theorems~\ref{thm:nonbt}--\ref{thm:boot.expo} that the estimation error of the
block bootstrap distribution, namely
\[\prob\big(\hat{T}^*_{b,\ell,k_1}\le \cdot\,\big|X_1,\ldots,X_n\big)-\prob(T_h\le \cdot\,),\]
has an order
\begin{multline}
O_p\Big(n^{-(\beta-1)/(4\beta+2)}h^{-(\beta+1)/(4\beta+2)}+h^{(\beta-2)/\beta}\\
+n^{-1/2+\delta}h^{-(3\beta+3)/(2\beta-2)+g_0(\beta)}
+k_1^{(\beta-2)/\beta}\\+b^{-1/2}\ell^{-1/2+\delta}k_1^{-3/2+3g_1(\beta)}+n^{-1/2+\delta}b^{-1/2}\ell^{1/2}k_1^{-3(1-g_2(\beta))/2}\Big)
\label{boot.err}
\end{multline}
under the conditions of Theorems~\ref{thm:nonbt} and \ref{thm:boot}, and
\begin{equation}
O_p\Big(h\log n+(nh)^{-1/2}+k_1\log\ell+n^{-1/2}\ell^{1/2}+(b\ell k_1)^{-1/2}\Big)
\label{boot.err.expo}
\end{equation}
under the conditions of Theorems~\ref{thm:nonbt.expo} and \ref{thm:boot.expo}.

The main goal here is to establish conditions on $(b,\ell,k_1)$ under which
the orders (\ref{boot.err}) and (\ref{boot.err.expo}) are minimised, respectively.

Denote by $\beta_1$ the solution to the equation
\[ 3g_1(\beta_1)=(\beta_1-2)/\beta_1,\]
which yields $\beta_1\approx 2.216$.
Define, for $\beta>2$,
\[ b_{min}(\beta)=\begin{cases}
(\beta-1) (3\beta-2 - 3 \beta g_1(\beta))/(5 \beta^2 - 5 \beta-4),&2<\beta\le\beta_1,\\
2\beta(\beta-1)/(5 \beta^2 - 5 \beta-4),&\beta>\beta_1,
\end{cases}\]
and
\[ b_{max}(\beta)=(4\beta^2-4\beta-4)/(5\beta^2-5\beta-4). \]

\begin{thm}
\label{thm:boot1}
Assume that $\alpha(t)=O(t^{-\beta})$ as $t\rightarrow\infty$, for some $\beta>2$.
Suppose that $b=O(n^{b_{max}(\beta)-\delta})$ and $n^{b_{min}(\beta)+2\delta}=O(b)$, for some $\delta>0$.
Let $y\in\mathbb{R}$ be fixed. Then
\begin{eqnarray*}
\lefteqn{\prob\big(\hat{T}^*_{b,\ell,k_1}\le y\big|X_1,\ldots,X_n\big)}\\
&=&
\Phi\left(\dfrac{y-n^{1/2}h^{5/2}f''(x_0)\mu_2/2}{\sqrt{f(x_0)\nu_2}}\right)
+o_p\left(n^{-(\beta-1)(\beta-2)/(5\beta^2-5\beta-4)}\right),
\end{eqnarray*}
where $(\ell,k_1)$ satisfy the following conditions.
\begin{itemize}
\item[(i)] If $\beta\le\beta_1$, then
\begin{gather*}
\ell\propto\min\left\{b^{1+\beta/(2 -3\beta + 3\beta g_1(\beta))},n/b\right\}\\
\text{and\ \ }k_1\propto n^{\delta'}\max\left\{\ell/b,(b\ell)^{-\beta/(5\beta-4-6\beta g_1(\beta))},\ell^{-1}\right\}
\end{gather*}
for some
$\displaystyle\delta'\in\Bigg(0,\:\min\left\{
\frac{2 \beta (2 \beta - 4 + 3 \beta (\beta - 1) g_1(\beta))}{(5 \beta^2 - 5 \beta - 4) (5 \beta - 4 -
   6 \beta g_1(\beta))},\,\delta/2\right\}\Bigg)$.
\item[(ii)] If $\beta>\beta_1$, then $\ell\propto\min\left\{b^{1/2},n/b\right\}$ and $k_1\propto n^{\delta'}\ell^{-1}$ for
some $\delta'\in(0,\delta/2)$.
\end{itemize}
\end{thm}
We see, by noting (\ref{min.order}), that the error rate
(\ref{boot.err}) reduces to
\begin{multline}
O_p\Big(n^{-(\beta-1)/(4\beta+2)}h^{-(\beta+1)/(4\beta+2)}\\+h^{(\beta-2)/\beta}
+n^{-1/2+\delta}h^{-(3\beta+3)/(2\beta-2)+g_0(\beta)}\Big)
\label{eq:opterr}
\end{multline}
under the conditions on $(b,\ell,k_1)$ given in Theorem~\ref{thm:boot1}. Note by
Theorem~\ref{thm:nonbt} that if we set the mean and variance to be the ``true'' values of the asymptotic mean and variance
of $T_h$ in its normal approximation, the resulting error rate has an order the same as (\ref{eq:opterr}).

\begin{thm}
\label{thm:boot1.expo}
Assume that $\alpha(t)=O(e^{-Ct})$ as $t\rightarrow\infty$, for some $C>0$.
Let $y\in\mathbb{R}$ be fixed and $\{L_n\}$ be any
positive, slowly varying, sequence converging to 0. Then
\begin{align*}
\prob\big(\hat{T}^*_{b,\ell,k_1}\le y\big|X_1,\ldots,X_n\big)&=\Phi\left(\dfrac{y-n^{1/2}h^{5/2}f''(x_0)\mu_2/2}{\sqrt{f(x_0)\nu_2}}\right) \\
& \quad\quad  + O_p\big\{n^{-1/3}(\log n)^{1/3}L_n^{-1}\big\},
\end{align*}
where $\ell\propto n^{1/3}(\log n)^{2/3}L_n^{1/2}$, $b\propto n/\ell$  and
$k_1\propto\big(\ell L_n^{1/2}\big)^{-1}$.
\end{thm}
Under the choices of $(b,\ell,k_1)$ given in Theorem~\ref{thm:boot1.expo}, the error rate
(\ref{boot.err.expo}) reduces to
\begin{equation}
\begin{cases}
O_p\big\{h\log n+(nh)^{-1/2}\big\},&h=O\big\{n^{-1/3}(\log n)^{-2/3}L_n^2\big\}\\
&
\text{\ \ or\ }n^{-1/3}(\log n)^{-2/3}L_n^{-1}=O(h),\\[0.5ex]
O_p\big\{n^{-1/3}(\log n)^{1/3}L_n^{-1}\big\},& \text{otherwise.}
\end{cases}
\label{eq:opterr.expo}
\end{equation}
Note by
Theorem~\ref{thm:nonbt.expo} that (\ref{eq:opterr.expo}) is equivalent to
the error
rate of normal approximation based on the ``true'' asymptotic mean and variance of $T_h$, except
when $h$ has an order within a slowly-varying factor from $n^{-1/3}(\log n)^{-2/3}$,
in which case (\ref{eq:opterr.expo}) exceeds the normal approximation error rate
by a slowly-varying factor.

\begin{remark}
In view of Theorem~\ref{thm:boot1} we may in practice set
\begin{equation}
b\propto n^{b_0},\;\;\ell\propto n^{\min\{b_0/2,\,1-b_0\}},\;\;k_1\propto n^{-\min\{b_0/2,\,1-b_0\}+\delta/2},
\label{prac}
\end{equation}
for some $b_0\in\big(b_{min}(\beta_1)+2\delta,b_{max}(\beta_1)-\delta\big)\approx\big(0.5689+2\delta, 0.7156-\delta\big)$
and any $\delta\in\big(0,\{b_{max}(\beta_1)-b_{min}(\beta_1)\}/3\big)\approx(0,0.04888)$. The setting (\ref{prac}) does not require knowledge of
$\beta$, satisfies conditions (ii) in Theorem~\ref{thm:boot1}, and minimises the error rate expressed in (\ref{boot.err}) for
$\beta>\beta_1$.

Note that a small choice of $b_0$, e.g.\ $0.5689\lessapprox b_0<2/3$, reduces computational cost, for which each bootstrap series has length
$b\ell\propto n^{3b_0/2}=o(n)$. The choice $b_0\ge 2/3$ amounts to the standard MBB which sets $b=\lfloor n/\ell\rfloor$.
\end{remark}

\begin{remark}
 Under the conditions of Theorem~\ref{thm:boot1.expo}, the practical choice (\ref{prac}), with $b_0$ set to $2/3$, minimises the error rate (\ref{boot.err.expo})
 whenever 
\[h=O\big(n^{-1/3-\delta}(\log n)^{-2}\big)
\text{\ \ or\ \ }n^{-1/3+\delta/2}\log n=O(h);\]
otherwise it yields an error rate of order $O_p\big(n^{-1/3+\delta/2}\log n\big)$.
\end{remark}

\begin{remark}
Under the conditions of Theorem~\ref{thm:boot1.expo}, the subsampling approach, which sets $b=1$, also provides a consistent distribution
estimator. Minimising (\ref{boot.err.expo}) under this setting returns the optimal choices  $\ell\propto n^{3/5}(\log n)^{2/5}$ and $k_1\propto n^{-1/5}(\log n)^{-4/5}$, leading
to an error rate of order $O_p\big(h\log n+(nh)^{-1/2}+n^{-1/5}(\log n)^{1/5}\big)$, which is inferior to that given by (\ref{prac}) with
$b_0=2/3$.
\end{remark}

\begin{remark}
Denote by $n^{q(\beta)}$ the order of the error term contributed by the block bootstrap to (\ref{boot.err}),
that is
\[
k_1^{(\beta-2)/\beta}+b^{-1/2}\ell^{-1/2+\delta}k_1^{-3/2+3g_1(\beta)}+n^{-1/2
+\delta}b^{-1/2}\ell^{1/2}k_1^{-3(1-g_2(\beta))/2}.
\]
The following diagrams plot $q(\beta)$, up to an arbitrarily small constant, against $\beta$ on the log scale under the practical choices (\ref{prac}),
when $b_0$ is set to
$2/3$ (blue) and $0.569$ (red), respectively. The exponent in (\ref{min.order}), that is
$-(\beta-1)(\beta-2)/(5\beta^2-5\beta-4)$, is shown for comparison (brown). The left and right diagrams correspond
respectively to the cases where $\beta\le\beta_1$ and $\beta>\beta_1$.
\end{remark}

\begin{center}
\includegraphics[scale=0.35]{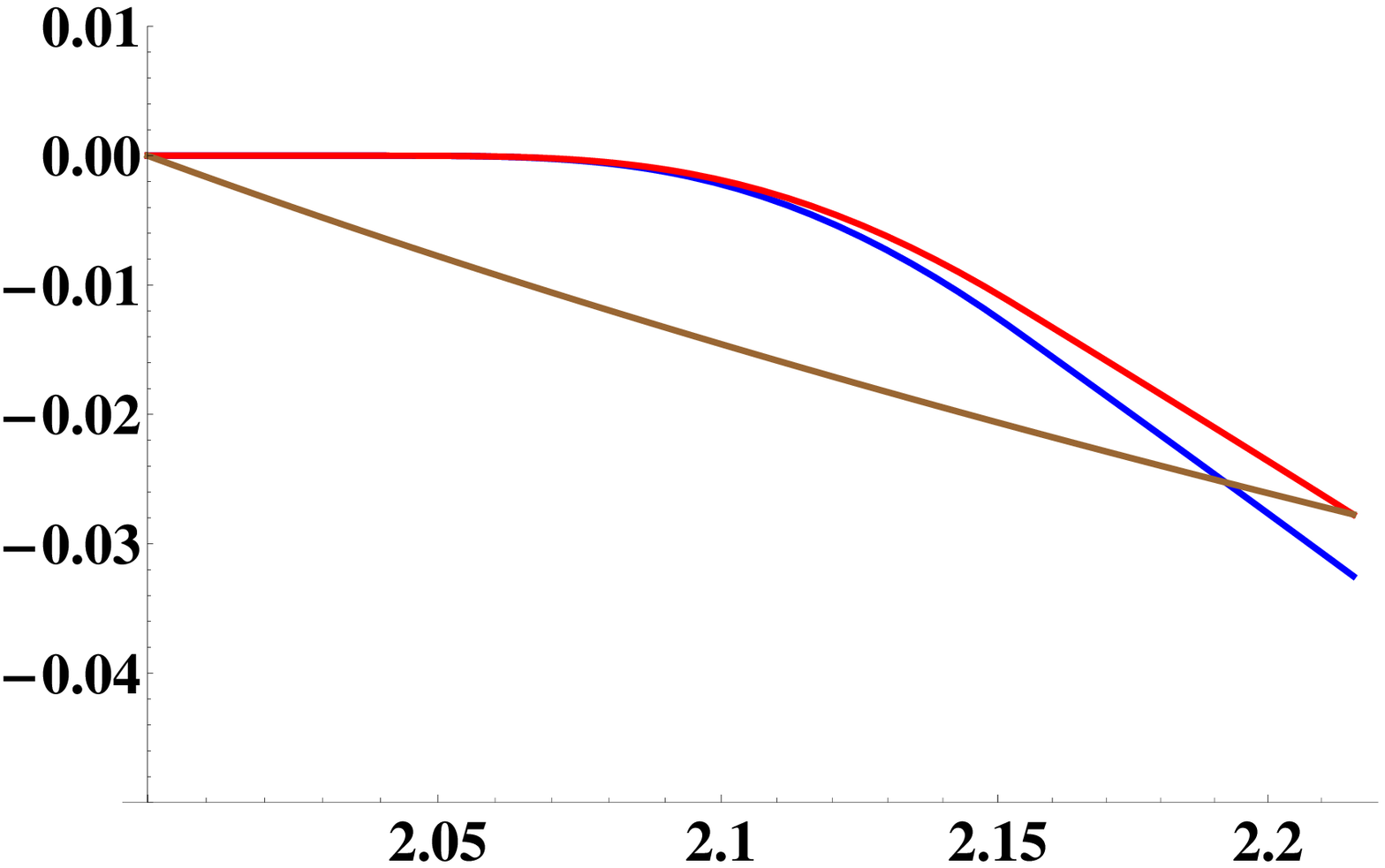} \includegraphics[scale=0.35]{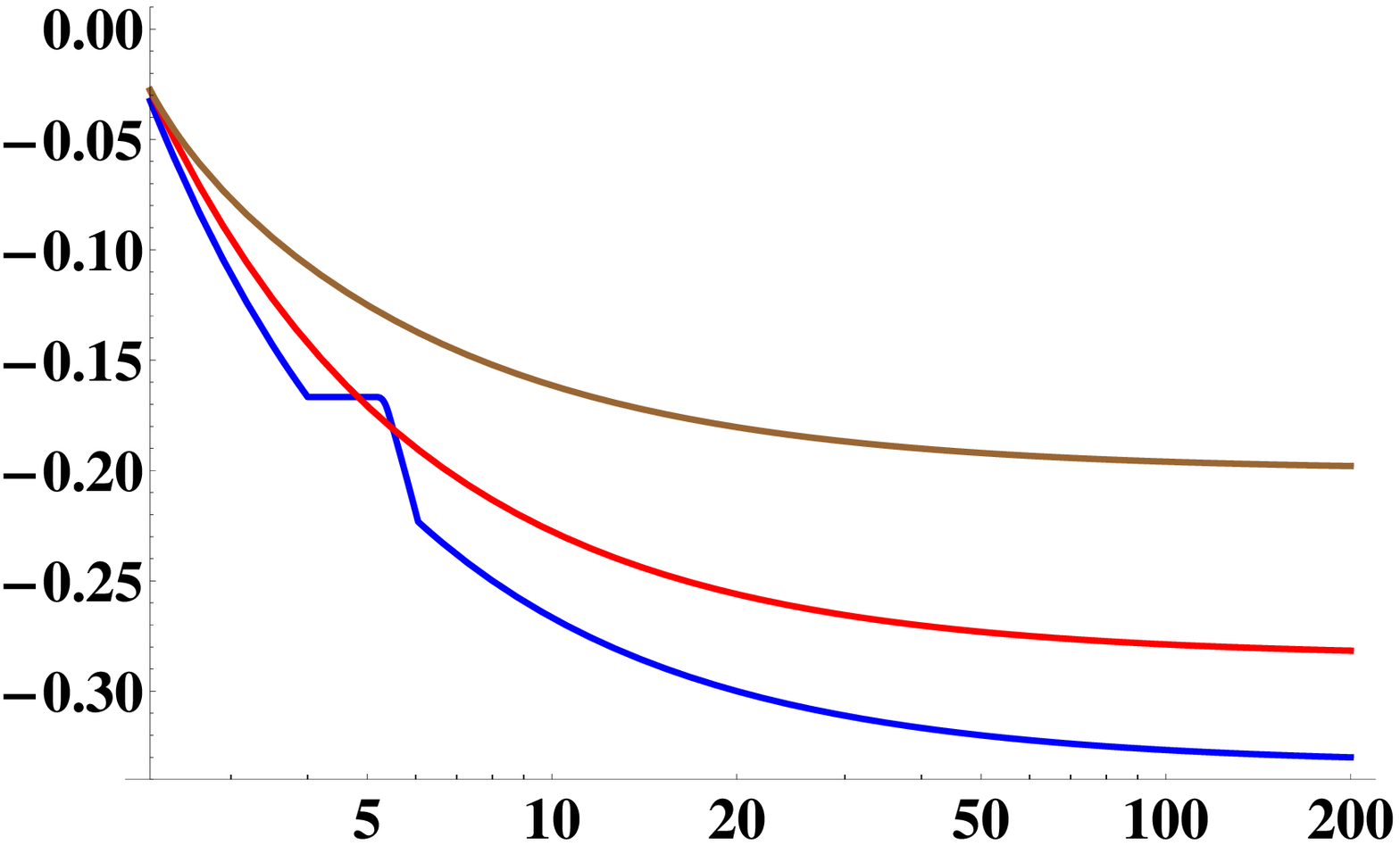}
\end{center}

\subsection{Choice of $(b,\ell)$ for NBC}
We see from Theorems~\ref{thm:nonbt}--\ref{thm:boot.expo} that
\[\prob\big(\tilde{T}^*_{b,\ell}\le \cdot\,\big|X_1,\ldots,X_n\big)-\prob(T_h\le \cdot\,)\]
has an order
\begin{multline}
O_p\Big\{n^{-(\beta-1)/(4\beta+2)}h^{-(\beta+1)/(4\beta+2)}+h^{(\beta-2)/\beta}\\+
n^{-1/2+\delta}h^{-(3\beta+3)/(2\beta-2)+g_0(\beta)}
+n^{-1/2}(b\ell)^{1/2}\\+\big(nh^5/(b\ell)\big)^{(\beta-2)/(5\beta)}
+n^\delta\big(nh^5)^{(6g_1(\beta)-3)/10}(b\ell)^{-(1+3 g_1(\beta))/5}\Big\}
\label{boot.err.nocorr}
\end{multline}
under the conditions of Theorems~\ref{thm:nonbt} and \ref{thm:boot}, and
\begin{multline}
O_p\Big\{h\log n+(nh)^{-1/2}+
n^{1/5}h(b\ell)^{-1/5}\log\ell\\+n^{-1/10}h^{-1/2}(b\ell)^{-2/5}+n^{-1/2}(b\ell)^{1/2}\Big\}
\label{boot.err.expo.nocorr}
\end{multline}
under the conditions of Theorems~\ref{thm:nonbt.expo} and \ref{thm:boot.expo}.

The following theorem establishes optimal choices of $(b,\ell)$ which minimise
(\ref{boot.err.nocorr}). 

\begin{thm}
\label{thm:boot2}
Assume the conditions of Theorem~\ref{thm:nonbt} and that 
\[b\ell=O(n),\quad b^{-4}\ell^{6}=O(nh^5)\quad\text{and}\quad n^{-1}\ell+ n^{-1}h^{-5}b\ell^{-4}=O(n^{-\epsilon}),\]
for some arbitrarily small $\epsilon>0$.
Let $y\in\mathbb{R}$ be fixed. Then,  for any arbitrarily small  $\delta>0$,
\begin{multline*}
\prob\big(\tilde{T}^*_{b,\ell}\le y\big|X_1,\ldots,X_n\big)-\Phi\left(\dfrac{y-n^{1/2}h^{5/2}f''(x_0)\mu_2/2}{\sqrt{f(x_0)\nu_2}}\right)\\
=O_p\Big\{n^{-1/2}(b\ell)^{1/2}+\big(nh^5/(b\ell)\big)^{(\beta-2)/(5\beta)}\\
+n^\delta\big(nh^5)^{(6g_1(\beta)-3)/10}(b\ell)^{-(1+3 g_1(\beta))/5}\Big\}
\end{multline*}
has a minimum order
\[ O_p\Big\{h^{(5\beta-10)/(7\beta-4)}+
\big(n^{-1 + 2\delta}h^{6g_1(\beta)-3}\big)^{5/(12 g_1(\beta)+14)}
+n^{-5/4+\epsilon/2}h^{-15/4}
\Big\},\]
attained by setting
\begin{itemize}
\item[(i)] 
$\ell\propto (nh^5)^{-1/2}n^{2\epsilon/5}$ and $b\propto n^{-1 +3\epsilon/5}h^{-5}$ if
\[h=O\Big(\min\big\{n^{-\frac{(7\beta-4)(5-2\epsilon)}{125\beta-100}},
n^{ -\frac{30 g_1(\beta)+25+20\delta-\epsilon(12 g_1(\beta)+14)}{150g_1(\beta)+75}}\big\}\Big);
\]

\item[(ii)]  
 $\ell=O\big(
n^{1/2}h^{(15\beta-20)/(14\beta-8)}\big)$, $b\ell\propto nh^{(10\beta-20)/(7\beta-4)}$
and $n^{\epsilon/5}h^{-5\beta/(7\beta-4)}=O(\ell)$
 if 
\[ h^{-1}=o\Big(\min\big\{n^{\frac{(7\beta-4)(5-2\epsilon)}{125\beta-100}}, n^{\frac{(7 \beta-4)(1-2\delta)}{35\beta -40 -  30\beta g_1(\beta)}}\big\}
\Big);\]

\item[(iii)] 
$\ell=O\Big\{\big(n^{6g_1(\beta)+3 + 8\delta}h^{30 g_1(\beta)-5}\big)^{1/(12 g_1(\beta)+14)}\Big\}$, 
$n^{\epsilon/5}\big(n^{1-2\delta}h^{10}\big)^{-1/(6 g_1(\beta)+7)}=O(\ell)$ and
$b\ell\propto \big(n^{(6 g_1(\beta)+2+10\delta)}h^{30 g_1(\beta)-15}\big)^{1/(6g_1(\beta)+7)}$ otherwise.

\end{itemize}
\end{thm}
Under the conditions on $(b,\ell)$ given in Theorem~\ref{thm:boot2}(i), (ii) and (iii), the error rate
(\ref{boot.err.nocorr}) reduces to
\begin{multline}
O_p\Big\{n^{-(\beta-1)/(4\beta+2)}h^{-(\beta+1)/(4\beta+2)}+h^{(5\beta-10)/(7\beta-4)}\\
+\,
\big(n^{-1 + 2\delta}h^{6g_1(\beta)-3}\big)^{5/(12 g_1(\beta)+14)}
+n^{-5/4+\epsilon/2}h^{-15/4}\Big\},
\label{eq:opterr2}
\end{multline}
which is in general larger than (\ref{eq:opterr}). Thus, NBC
is in general inferior to EBC in terms of
their minimum error rates when $\alpha(t)$ decays at a polynomial rate with $\beta>2$.

\begin{thm}
\label{thm:boot2.expo}
Assume the conditions of Theorem~\ref{thm:nonbt.expo}.
Suppose that
\[b\ell=O(n)\quad\text{and}\quad n^{-1}\ell+ n^{-1}h^{-5}b\ell^{-4}=o(1).\]
Let $y\in\mathbb{R}$ be fixed. Then
\begin{multline*}
\prob\big(\tilde{T}^*_{b,\ell}\le y\big|X_1,\ldots,X_n\big)-
\Phi\left(\dfrac{y-n^{1/2}h^{5/2}f''(x_0)\mu_2/2}{\sqrt{f(x_0)\nu_2}}\right)\\
=O_p\Big\{n^{1/5}h(b\ell)^{-1/5}\log\ell+n^{-1/10}h^{-1/2}(b\ell)^{-2/5}+n^{-1/2}(b\ell)^{1/2}\Big\}
\end{multline*}
has a minimum order $O_p\big\{(nh)^{-5/18}+(h\log n)^{5/7}\big\}$,
attained by setting
\begin{equation}
\begin{cases}
b\ell\propto n^{4/9}h^{-5/9},\;
nh^{10}\ell^9\rightarrow\infty,
\text{\ if\ }h^{25}=O\big\{n^{-7}(\log n)^{-18}\big\},\\
b\ell\propto n(h\log n)^{10/7},
\;h^{5}(\log n)^{-2}\ell^7\rightarrow\infty,
\text{\ if\ }n^{7}(\log n)^{18}h^{25}\rightarrow\infty.
\end{cases}
\label{eq:opt.expo.nocorr}
\end{equation}
\end{thm}
We see from Theorem~\ref{thm:boot2.expo} that (\ref{eq:opt.expo.nocorr}) covers the subsampling case where $b=1$, provided that
we set $\ell\propto \max\big\{n^{4/9}h^{-5/9},n(h\log n)^{10/7}\big\}$.

Under the conditions (\ref{eq:opt.expo.nocorr}) on $(b,\ell)$ given in Theorem~\ref{thm:boot2.expo}, the error rate
(\ref{boot.err.expo.nocorr}) reduces to
\begin{equation}
O_p\big\{(nh)^{-5/18}+(h\log n)^{5/7}\big\},
\label{eq:opterr2.expo}
\end{equation}
which is larger than (\ref{eq:opterr.expo}). Thus, NBC
is again inferior to EBC under an exponential $\alpha$-mixing rate.

\begin{remark} 
In the absence of explicit bias correction, the scaled bias term, namely
\[(nh)^{1/2}\big\{
\expn\hat{f}_h(x_0)-f(x_0)\big\}=n^{1/2}h^{5/2}f''(x_0)\mu_2/2
+O\big(n^{1/2}h^{9/2}\big),\]
has the block bootstrap analogue
\[ (b\ell k)^{1/2}\Big\{\expn\big[\hat{f}^*_{b,\ell,k}\big|X_1,\ldots,X_n\big]
-\hat{f}_{k_3}(x_0)\Big\}\]
by construction of $\tilde{T}^*_{b,\ell}$,
where $k$ and $k_3$ satisfy (\ref{bias.cond.nocorr}).
The block bootstrap mean $\expn\big[\hat{f}^*_{b,\ell,k}\big|X_1,\ldots,X_n\big]$
is subject to a sampling variation of order $O_p\big((nk)^{-1/2}\big)$, which can only
be offset asymptotically by setting the length of the block bootstrap series $b\ell=o(n)$. It thus follows that the standard MBB,
which takes $b=\lfloor n/\ell\rfloor$, fails to consistently estimate the
scaled bias.
\end{remark}

\begin{remark} 
In view of Theorem~\ref{thm:boot2} we may in practice set,
by considering the limiting case $\beta\rightarrow\infty$,
\begin{equation}
\begin{cases}
\ell\propto (nh^5)^{-1/2}n^{2\epsilon},\;\;b\propto n^{-1 +3\epsilon}h^{-5},
&\text{if\ \ }h=O(n^{-7/25}),\\
b\ell\propto nh^{10/7},\;\;\ell=O(n^{1/2}h^{15/14}),\;\;
h^{5/7}\ell\rightarrow\infty,
&\text{if\ \ }n^{7/25}h\rightarrow\infty,
\end{cases}
\label{prac.nocorr}
\end{equation}
for an arbitrarily small $\epsilon>0$.

Suppose that (\ref{prac.nocorr}) holds for $(b,\ell)$. Then, under the
conditions of Theorem~\ref{thm:nonbt}, the error rate
(\ref{boot.err.nocorr}) reduces to
\begin{multline}
O_p\Big\{n^{-(\beta-1)/(4\beta+2)}h^{-(\beta+1)/(4\beta+2)}+h^{(\beta-2)/\beta}\\+
n^{-1/2+\delta}h^{-(3\beta+3)/(2\beta-2)+g_0(\beta)}
+n^{-5/4(1-2\epsilon)}h^{-15/4}\\+\big(n^{1-2\epsilon}h^5\big)^{(\beta-2)/(2\beta)}
+(nh^5)^{3 g_1(\beta)/2}n^{\delta-\epsilon(3g_1(\beta)+1)}\Big\}
\label{eq:opterr2.nocorr1}
\end{multline}
if $h=O\big(n^{-7/25}\big)$, and to
\begin{multline}
O_p\Big(n^{-(\beta-1)/(4\beta+2)}h^{-(\beta+1)/(4\beta+2)}\\+
h^{(5\beta-10)/(7\beta)}+n^{-1/2+\delta}h^{(30g_1(\beta)-25)/14}\Big)
\label{eq:opterr2.nocorr2}
\end{multline}
if $n^{7/25}h\rightarrow\infty$. Under the
conditions of Theorem~\ref{thm:nonbt.expo}, the error rate
(\ref{boot.err.expo.nocorr}) reduces to
\begin{equation}
\begin{cases}
O_p\big(n^{-5/4(1-2\epsilon)}h^{-15/4}\big),&\text{if\ }h=O\big(n^{-7/25}\big),\\
O_p\big(h^{5/7}\log n\big),&\text{if\ }n^{7/25}h\rightarrow\infty.
\end{cases}
\label{eq:opterr2.expo.nocorr}
\end{equation}
\end{remark} 

\begin{remark} 
Under the conditions of Theorem~\ref{thm:nonbt.expo}, a comparison between (\ref{eq:opterr2.expo}) and
(\ref{eq:opterr2.expo.nocorr})
shows that the error rate given by
the practical choice (\ref{prac.nocorr}) is slower, by at least a slowly-varying factor, than the error
rate  attained by the theoretically optimal choice (\ref{eq:opt.expo.nocorr}).
\end{remark}

\begin{remark}
Denote by $n^{\tilde{q}(\beta)}$ the order of the error term contributed by the block bootstrap to (\ref{boot.err.nocorr}),
that is
\[n^{-1/2}(b\ell)^{1/2}+\big(nh^5/(b\ell)\big)^{(\beta-2)/(5\beta)}
+n^\delta\big(nh^5)^{(6g_1(\beta)-3)/10}(b\ell)^{-(1+3 g_1(\beta))/5}.\]
The following diagrams plot $\tilde{q}(\beta)$, up to an arbitrarily small constant, against $\beta$ on the log scale, when
$(b,\ell)$ is set to be the optimal
choice specified in Theorem~\ref{thm:boot2} (blue) and the practical choice (\ref{prac.nocorr}) (red), respectively.
Three different orders of $h$ are considered: (i) $h\propto n^{-0.2}$,
(ii) $h\propto n^{-0.25}$ and (iii)  $h\propto n^{-0.331}$.
For comparison we plot also the exponent $q_0(\beta)$ (brown), where $n^{q_0(\beta)}$ denotes the order of (\ref{eq:opterr}),
which corresponds to the error
rate of normal approximation based on the ``true'' asymptotic mean and variance of $T_h$.
Note that a positive difference $\tilde{q}(\beta)-q_0(\beta)$ indicates an inflation of error induced by the block bootstrap scheme.
In case (iii), the blue and red lines coincide for $\beta>2.13658$ and are almost indistinguishable for $\beta\in(2,2.13658]$.
\end{remark} 

\begin{center}
\small
\begin{tabular}{cc}
(i) & (ii) \\
\includegraphics[scale=0.35]{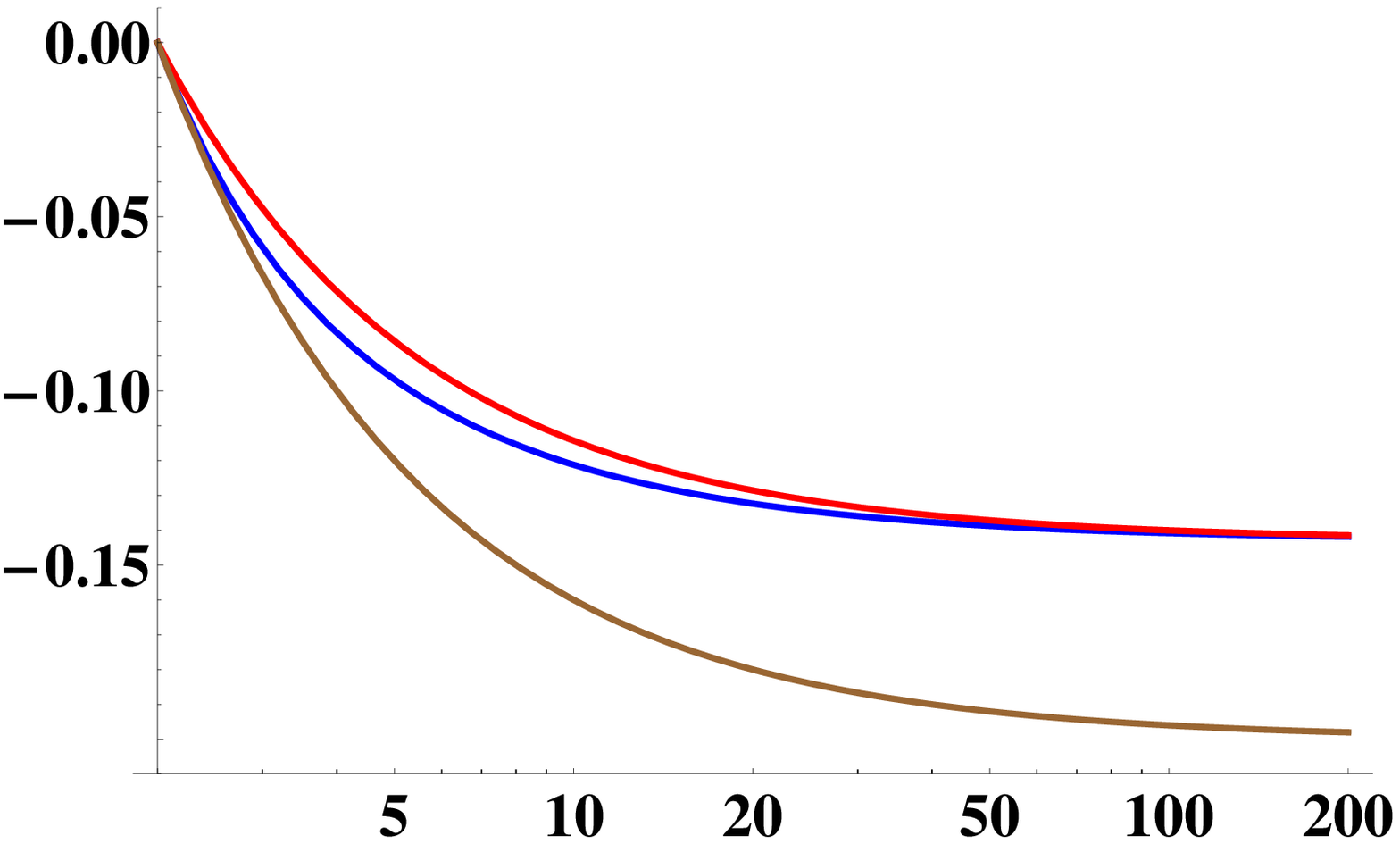} & \includegraphics[scale=0.35]{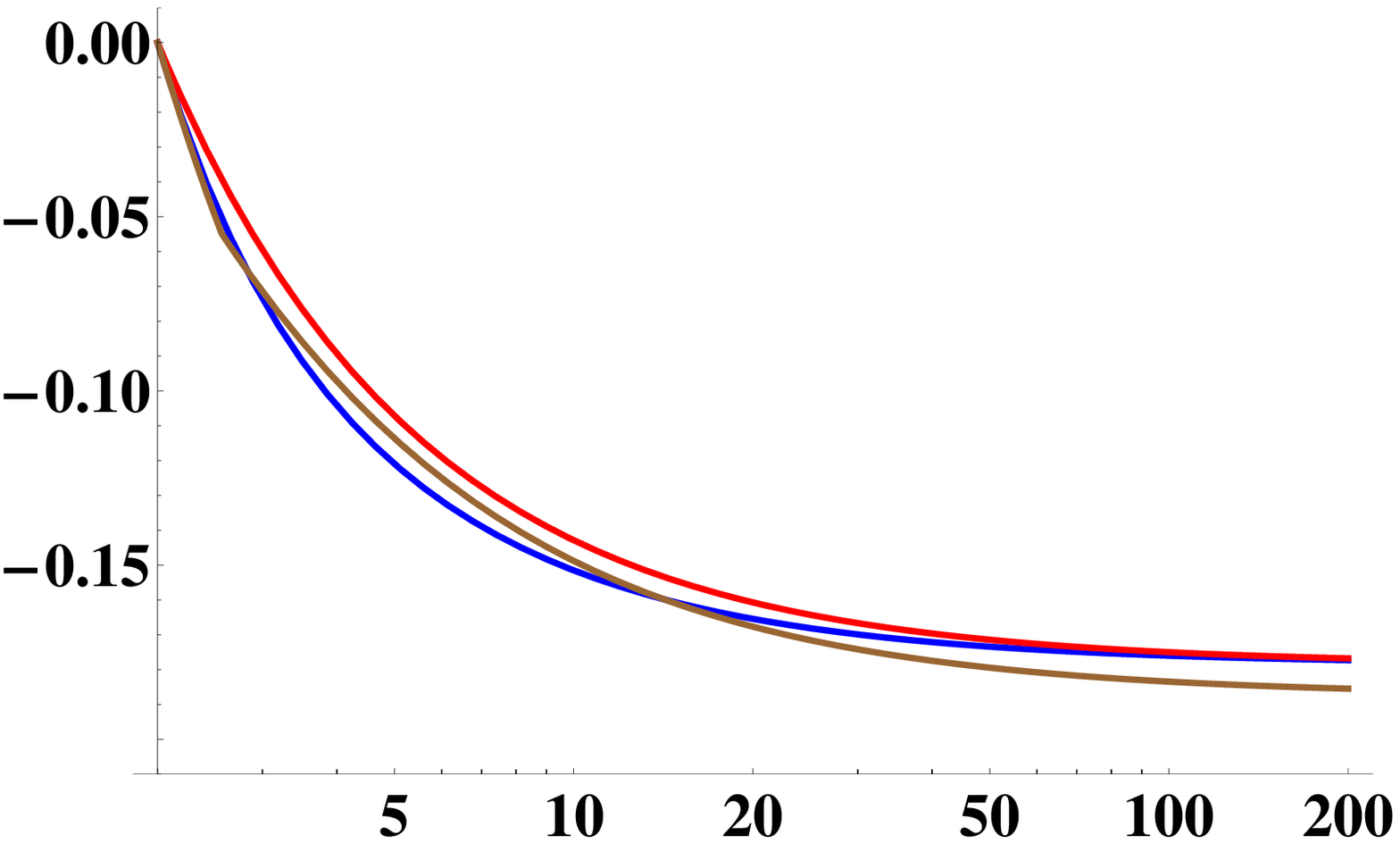}\\
\multicolumn{2}{c}{(iii)}\\
\multicolumn{2}{c}{
\includegraphics[scale=0.35]{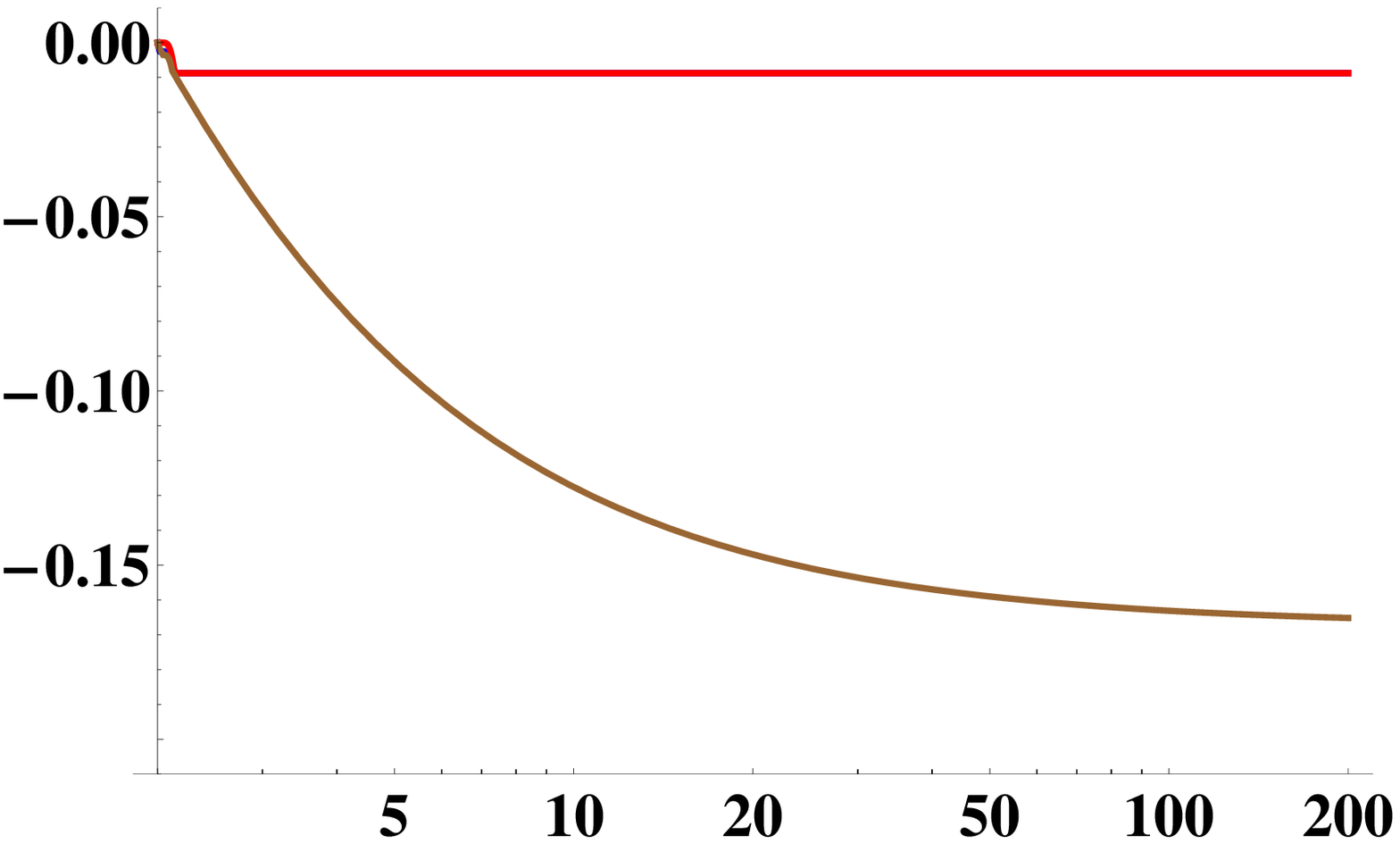}
}
\end{tabular}
\end{center}

\begin{remark}
\citet{Neumann:1998} studied the convergence rate of the kernel density estimator under weak dependence, utilizing arguments for iid data. For comparison with our results, his Theorem 3.3 gives an error term containing $h^{1/2}(log n)^{3/2}$. Taking $h$ to be of the conventional optimal order $n^{-1/5}$, the above term would be bigger than $n^{-1/10}$. 
\end{remark}

\subsection{Choice of $(b,\ell,k)$ for UNS}

We see from Theorems~\ref{thm:nonbt}--\ref{thm:boot.expo} that
\[\prob\big(\check{T}^*_{b,\ell,k}\le \cdot\,\big|X_1,\ldots,X_n\big)-\prob(T_h\le \cdot\,)\]
has an order
\begin{multline}
O_p\Big\{n^{-(\beta-1)/(4\beta+2)}h^{-(\beta+1)/(4\beta+2)}+h^{(\beta-2)/\beta}\\+
n^{-1/2+\delta}h^{-(3\beta+3)/(2\beta-2)+g_0(\beta)}+n^{1/2}h^{5/2}+
k^{(\beta-2)/\beta}\\+b^{-1/2}\ell^{-1/2+\delta}k^{-3/2+3g_1(\beta)}+n^{-1/2+\delta}b^{-1/2}\ell^{1/2}k^{-3(1-g_2(\beta))/2}\Big\}
\label{boot.err.undersm}
\end{multline}
under the conditions of Theorems~\ref{thm:nonbt} and \ref{thm:boot}, and
\begin{equation}
O_p\Big\{h\log n+(nh)^{-1/2}+n^{1/2}h^{5/2}+
k\log\ell+n^{-1/2}\ell^{1/2}+(b\ell k)^{-1/2}\Big\}
\label{boot.err.expo.undersm}
\end{equation}
under the conditions of Theorems~\ref{thm:nonbt.expo} and \ref{thm:boot.expo}.

Denote by $\beta_2$ the biggest solution to the equation
\[ 3g_2(\beta_2)=(\beta_2-4)/\beta_2,\]
which yields $\beta_2\approx 6.0538$.
The following theorem establishes optimal choices of $(b,\ell,k)$ which minimise
(\ref{boot.err.undersm}). 

\begin{thm}
	\label{thm:boot3}
	Assume the conditions of Theorem~\ref{thm:nonbt}.
	Let $y\in\mathbb{R}$ be fixed. Then,  for any arbitrarily small  $\delta>0$,
	\begin{multline*}
	\prob\big(\check{T}^*_{b,\ell,k}\le y\big|X_1,\ldots,X_n\big)-\Phi\left(\dfrac{y-n^{1/2}h^{5/2}f''(x_0)\mu_2/2}{\sqrt{f(x_0)\nu_2}}\right)\\
	=O_p\Big\{n^{1/2}h^{5/2}+
	k^{(\beta-2)/\beta}+b^{-1/2}\ell^{-1/2+\delta}k^{-3/2+3g_1(\beta)}\\+n^{-1/2+\delta}b^{-1/2}\ell^{1/2}k^{-3(1-g_2(\beta))/2}\Big\}
	\end{multline*}
	has a minimum order
	\begin{multline*}
	    O_p\Big\{n^{1/2}h^{5/2}+n^{-(\beta-2)/\{\beta (5-6 g_1(\beta))-4\}+\delta}\\
	+n^{-(\beta - 2)/(3\beta)+\delta}+n^{-2 (\beta-2)/\{\beta(7-3 g_2(\beta))-4\}+\delta}\Big\},
	    \end{multline*}
	attained by setting $b\propto n/\ell$, $\ell=k^{-1-\delta'}$ and
	\[k\propto\begin{cases}
n^{-\beta/\{\beta (5-6 g_1(\beta))-4\}},&2<\beta<\beta_1,\\
n^{-2 \beta/\{\beta(7-3 g_2(\beta))-4\}},&4<\beta<\beta_2,\\
n^{-1/3+\delta'},& \beta\in[\beta_1,4]\cup[\beta_2,\infty),
	\end{cases}\]
	for sufficiently small $\delta'>0$.
	\end{thm}
Under the conditions on $(b,\ell,k)$ given in Theorem~\ref{thm:boot3}, the error rate
(\ref{boot.err.undersm}) is in general larger than (\ref{eq:opterr}). Thus, UNS
is in general inferior to EBC under a polynomial $\alpha$-mixing rate with $\beta>2$.

\begin{thm}
	\label{thm:boot3.expo}
	Assume the conditions of Theorem~\ref{thm:nonbt.expo}.
	Let $y\in\mathbb{R}$ be fixed. Then
	\begin{multline*}
	\prob\big(\check{T}^*_{b,\ell,k}\le y\big|X_1,\ldots,X_n\big)-
	\Phi\left(\dfrac{y-n^{1/2}h^{5/2}f''(x_0)\mu_2/2}{\sqrt{f(x_0)\nu_2}}\right)\\
	=O_p\Big\{n^{1/2}h^{5/2}+
	k\log\ell+n^{-1/2}\ell^{1/2}+(b\ell k)^{-1/2}\Big\}
	\end{multline*}
	has a minimum order $O_p\big\{n^{1/2}h^{5/2}+n^{-1/3}(\log n)^{1/3}L_n^{-1}\big\}$,
	attained by setting $b\propto n/\ell$, $\ell\propto (kL_n^3)^{-1}$ and
	$k\propto n^{-1/3}(\log n)^{-2/3}L_n^{-1}$, for any
positive, slowly varying, sequence $\{L_n\}$ converging to 0.	
\end{thm}
Under the conditions on $(b,\ell,k)$ given in Theorem~\ref{thm:boot3.expo}, the error rate
(\ref{boot.err.expo.undersm}) reduces to
\begin{equation}
O_p\big\{h\log n+(nh)^{-1/2}+n^{1/2}h^{5/2}+n^{-1/3}(\log n)^{1/3}L_n^{-1}\big\},
\label{eq:opterr3.expo}
\end{equation}
which is never smaller than (\ref{eq:opterr.expo}). In fact, it is bigger than
(\ref{eq:opterr.expo}) if $n^{-1/3}(\log n)^{2/3}=o(h)$.  Thus, UNS is in general never superior 
to EBC under an exponential $\alpha$-mixing rate.

Note that
(\ref{eq:opterr3.expo}) is smaller than the optimal rate
(\ref{eq:opterr2.expo}) given by NBC if and only if
$h=o\big\{n^{-7/25}(\log n)^{2/5}\big\}$. In other words, if both methods are
tuned optimally, UNS is preferable to NBC when $h=o\big\{n^{-7/25}(\log n)^{2/5}\big\}$. 

The following table summarises the optimal error rates under the conditions of Theorem~\ref{thm:nonbt.expo} and different asymptotic regimes of $h$.

\vspace{2ex}\hspace{-10ex}\renewcommand{\arraystretch}{1.2}{\footnotesize
\begin{tabular}{|c|ccc|}
	\hline
	& \multicolumn{3}{c|}{Method} \\
	Range of $h$ & EBC & NBC & UNS \\ \hline
		$h\propto n^{-1/5}$ & 
	{\color{blue}$h\log n$} & {$(h\log n)^{5/7}$} & {\color{red}(inconsistent)} \\ 
	$n^{-7/25}(\log n)^{2/5}\preceq h\prec n^{-1/5}$ & 
	{\color{blue}$h\log n$} & {$(h\log n)^{5/7}$} & {\color{red}$n^{1/2}h^{5/2}$} \\ 
$n^{-7/25}(\log n)^{-18/25}\preceq h\preceq n^{-7/25}(\log n)^{2/5}$ & 
{\color{blue}$h\log n$} & {\color{red}$(h\log n)^{5/7}$} & {$n^{1/2}h^{5/2}$} \\ 	
$n^{-1/3}(\log n)^{2/3}\preceq h\preceq n^{-7/25}(\log n)^{-18/25}$ & 
{\color{blue}$h\log n$} & {\color{red}$(nh)^{-5/18}$} & {$n^{1/2}h^{5/2}$} \\ 
$n^{-1/3}(\log n)^{-2/3}L_n^{-1}\preceq h\preceq n^{-1/3}(\log n)^{2/3}$ & 
{\color{blue}$h\log n$} & {\color{red}$(nh)^{-5/18}$} & {\color{blue}$h\log n$} \\
$n^{-1/3}(\log n)^{-2/3}L_n^{2}\preceq h\preceq n^{-1/3}(\log n)^{-2/3}L_n^{-1}$ & 
{\color{blue}$n^{-1/3}(\log n)^{1/3}L_n^{-1}$} & {\color{red}$(nh)^{-5/18}$} & {\color{blue}$n^{-1/3}(\log n)^{1/3}L_n^{-1}$} \\
$n^{-1}\prec h\preceq n^{-1/3}(\log n)^{-2/3}L_n^{2}$ & 
{\color{blue}$(nh)^{-1/2}$} & {\color{red}$(nh)^{-5/18}$} & {\color{blue}$(nh)^{-1/2}$} \\	\hline
\multicolumn{4}{l}{\footnotesize Notation: $a_n\prec b_n$ iff $a_n=o(b_n)$, $a_n\preceq b_n$ iff $a_n=O(b_n)$.}\\
\multicolumn{4}{l}{\footnotesize  Best rate (blue), worst rate (red) under each $h$ regime.}
	\end{tabular}
}

\subsection{Comparison with i.i.d.\ case}
Consider the case where $(X_1,\ldots,X_n)$ are i.i.d.\ and applications of EBC, NBC and UNS to
estimate the distribution of $T_h$, with block length $\ell$ set to 1.

Define $\sigma^2_h=(nh)\text{Var}\big(\hat{f}_h(x_0)\big)$. For completeness we provide i.i.d.\ analogues of Theorems~\ref{thm:nonbt.expo}
and \ref{thm:boot.expo}. For brevity, proofs are omitted.
\begin{thm}
	\label{thm:nonbt.iid}
	Suppose that $nh\rightarrow\infty$ and $h=O(n^{-1/5})$.
	Then, for any fixed $y\in\mathbb{R}$,
	\begin{eqnarray*}
		\lefteqn{
			\prob\Big((nh)^{1/2}\big\{\hat{f}_h(x_0)-f(x_0)\big\}\le y\Big)}\\&=&
		\Phi\left(\dfrac{y-n^{1/2}h^{5/2}f''(x_0)\mu_2/2}{\sigma_h}\right)+O\big\{n^{1/2}h^{9/2}+(nh)^{-1/2}\big\}.
	\end{eqnarray*}
\end{thm}
\begin{thm}
	\label{thm:boot.iid}
	Suppose that $b=O(n)$, $b\rightarrow\infty$,
	$(bk_1)^{-1}+k_1=o(1)$ and $(nk_j)^{-1}+k_j=o(1)$, $j=2,3$.
	Let $y\in\mathbb{R}$ be fixed. Then
	\begin{multline*}
	\prob\big(T^*_{b,1,\tau,k_1,k_2,k_3}\le y\big|X_1,\ldots,X_n\big)=
	\Phi\left(\dfrac{y-\tau(k_2^2-k_3^2)f''(x_0)\mu_2/2}{\sigma_{k_1}}\right)\\
	+\,O_p\Big\{(bk_1)^{-1/2}+\tau\big[k_2^4+k_3^4+(nk_2)^{-1/2}+(nk_3)^{-1/2}\big]\Big\}.
	\end{multline*}
\end{thm}

\subsubsection{Choice of $(b,k_1)$ for EBC}
It follows 
from Theorems~\ref{thm:nonbt.iid} and \ref{thm:boot.iid} that the estimation error of the
block bootstrap distribution, namely
\[\prob\big(\hat{T}^*_{b,1,k_1}\le \cdot\,\big|X_1,\ldots,X_n\big)-\prob(T_h\le \cdot\,),\]
has an order
\[
O_p\Big\{(nh)^{-1/2}+n^{5/18}h^{5/2}+(bk_1)^{-1/2}+(k_1^2-h^2)\Big\},
\]
which has a minimum order
\begin{equation}
O_p\Big\{(nh)^{-1/2}+n^{5/18}h^{5/2}\Big\}
\label{boot.err.iid}
\end{equation}
if we set $k_1=h$ and $b\propto n$. The latter setting leads to Hall's (1992) explicit bias correction method.

\subsubsection{Choice of $b$ for NBC}
It follows 
from Theorems~\ref{thm:nonbt.iid} and \ref{thm:boot.iid} that 
\[\prob\big(\tilde{T}^*_{b,1}\le \cdot\,\big|X_1,\ldots,X_n\big)-\prob(T_h\le \cdot\,)\]
has an order
\[
O_p\big(b^{1/2}n^{-1/2}+n^{-1/10}b^{-2/5}h^{-1/2}\big),
\]
which has a minimum order
\begin{equation}
O_p\big\{(nh)^{-5/18}\big\},
\label{boot.err.iid.nocorr}
\end{equation}
if we set $b\propto n^{4/9}n^{-5/9}=o(n)$.
Note also that $k,k_3\propto(n/b)^{1/5}h$ has a bigger order than $h$.
 The method is therefore analogous to a $b$-out-of-$n$ bootstrap procedure 
coupled with oversmoothing.

\subsubsection{Choice of $(b,k)$ for UNS}
It follows 
from Theorems~\ref{thm:nonbt.iid}, \ref{thm:boot.iid} and the additional assumption $h=o(n^{-1/5})$ that 
\[\prob\big(\check{T}^*_{b,1,k}\le \cdot\,\big|X_1,\ldots,X_n\big)-\prob(T_h\le \cdot\,)\]
has an order
\[
O_p\big\{(bk)^{-1/2}+n^{1/2}h^{5/2}+(nh)^{-1/2}+(k^2-h^2)\big\},
\]
which has a minimum order
\begin{equation}
O_p\big\{n^{1/2}h^{5/2}+(nh)^{-1/2}\big\},
\label{boot.err.iid.undersm}
\end{equation}
if we set $k=h$ and $b\propto n$. This corresponds to Hall's (1992) undersmoothing method.

The following table summarises the optimal error rates under the i.i.d.\ assumption and different asymptotic regimes of $h$.

\vspace{2ex}\renewcommand{\arraystretch}{1.2}{\footnotesize
\begin{tabular}{|c|ccc|}
	\hline
	& \multicolumn{3}{c|}{Method} \\
	Range of $h$ & EBC & NBC & UNS \\ \hline
	$h\propto n^{-1/5}$ & 
	{\color{blue}$n^{5/18}h^{5/2}$} & {$(nh)^{-5/18}$} & {\color{red}(inconsistent)} \\ 
	$n^{-7/27}\preceq h\prec n^{-1/5}$ & 
	{\color{blue}$n^{5/18}h^{5/2}$} & {$(nh)^{-5/18}$} & {\color{red}$n^{1/2}h^{5/2}$} \\ 
	$n^{-7/25}\preceq h\preceq n^{-7/27}$ & 
	{\color{blue}$(nh)^{-1/2}$} & {$(nh)^{-5/18}$} & {\color{red}$n^{1/2}h^{5/2}$} \\ 
		$n^{-1/3}\preceq h\preceq n^{-7/25}$ & 
	{\color{blue}$(nh)^{-1/2}$} & {\color{red}$(nh)^{-5/18}$} & {$n^{1/2}h^{5/2}$} \\ 	
	$n^{-1}\prec h\preceq n^{-1/3}$ & 
{\color{blue}$(nh)^{-1/2}$} & {\color{red}$(nh)^{-5/18}$} & {\color{blue}$(nh)^{-1/2}$} \\ 	\hline
\multicolumn{4}{l}{\footnotesize Notation: $a_n\prec b_n$ iff $a_n=o(b_n)$, $a_n\preceq b_n$ iff $a_n=O(b_n)$.}\\
\multicolumn{4}{l}{\footnotesize  Best rate (blue), worst rate (red) under each $h$ regime.}
\end{tabular}
}

\section{Simulation Study}
\label{sec:simulation}

We provide illustration of some of the methodological conclusions of the preceding theory by considering the density estimation problem when the data sample $\{X_1, \ldots, X_n\}$ arises from the $ARMA(1,1)$ model
\[
X_{t}-\phi X_{(t-1)}=\epsilon_{t}+\theta \epsilon_{(t-1)},
\]
with the $\epsilon_{t}$ independent, identically distributed $N(0,1)$, and the coefficients fixed as $\phi=0.4, \theta=0.3$. This ARMA model satisfies the strong mixing condition with an exponential rate of decay for the mixing coefficients \citep[Example 6.1]{Lahiri:2003}. Throughout our simulations, the value of $X_{0}$ was sampled randomly from the marginal distribution $N(0, 1+\frac{(\theta+\phi)^{2}}{1-\phi^{2}})$, and $\epsilon_{0}$ is sampled from $N(0,1)$. The true density $f(\cdot)$ being estimated is therefore that of a normal distribution $N(0,1.5833)$. We consider estimation of $P(x_0,y) \equiv \prob\Big((nh)^{1/2}\big\{\hat{f}_h(x_0)-f(x_0)\big\}\le y\Big)$. Throughout, we consider density estimation based on the Epanechnikov kernel function $K(u)=\frac{3}{4}(1-u^2), |u| \le 1$. All bootstrap estimators are based on drawing 10,000 block bootstrap samples from a given sample, and accuracy is measured in terms of mean squared error in estimation of the true value of the probability of interest, over a series of 10,000 replications from the $ARMA(1,1)$ model.

We consider the three block bootstrap estimators of $P(x_0,y)$ discussed previously, defined 
in terms
\[
\prob \big(\hat T^*_{b,\ell,\tau, k_1, k_2, k_3} \le y\big| X_1,\ldots,X_n\big).
\]
We denote by $NBC$ the estimator without explicit bias correction: recall this is defined by setting $k_1=k_2=k, \tau=(b\ell k)^{1/2}$. We fix the value of $c_0$ as $c_0=0.5$, then fix $k$ and $k_3$ as specified by \eqref{bias.cond.nocorr}. This estimator $NBC$ then requires only specification of $(b, \ell)$. We denote by $EBC$ the estimator which makes an explicit bias correction, where we specify $k_2=c_2 n^{-1/9}$, $k_3=c_0k_2$, $\tau=(1-c_0^2)^{-1}n^{1/2}h^{5/2}k_2^{-2}$, with, as before, $c_0=0.5$. This estimator requires specification, in addition to $(b, \ell)$, of the two tuning constants, $k_1$ and $c_2$, which defines $k_2$. The third estimator, which avoids the explicit bias correction by undersmoothing, we denote by $UNS$, and corresponds to setting $\tau=0$ in the definition of $EBC$, leaving an estimator depending on just the one tuning constant $k_1$ in addition to $(b, \ell)$.

We consider first sample size $n=100$, and estimation for $x_0=1.0, y=0.15$. Two values of the kernel estimator bandwidth $h$ are considered, $h=0.625, 1.80$. The former value corresponds to the bandwidth for which, by Theorem~\ref{thm:nonbt.expo}, $P(x_0,y)$ is most accurately estimated by the normal approximation
\[
\Phi\left(\dfrac{y-n^{1/2}h^{5/2}f''(x_0)\mu_2/2}{\sqrt{f(x_0)\nu_2}}\right).
\]
Note that this latter quantity is, in practice, unavailable since $f(x_0)$ is unknown, and the object of the density estimation. Asymptotically, the bandwidth $h$ specified in this way is of order $n^{-1/3}$. The value $h=1.80$ is the bandwidth which minimises the mean squared error of estimation, $E\{\hat f_h(x_0)-f(x_0)\}^2$. Asymptotically, the bandwidth minimising the mean squared error of density estimation is of order $n^{-1/5}$.  Note that the bias of the estimator $\hat f(x_0)$ of $f(x_0)$ increases with $h$, and that for the case $n=100$ considered here bias is greater, by a factor around 10, at $h=1.80$, compared to $h=0.625$, where it is negligible. In detail, bias is -0.0021159 at $h=0.625$, and -0.0179130 at $h=1.80$. An important methodological alert of the theory presented in Section~\ref{sec:theory} is that the estimator $UNS$ is inconsistent if the density estimator bandwidth $h$ is of order $n^{-1/5}$, but ought to work well, with similar asymptotics to $EBC$, if the density estimator is undersmoothed, in particular in the case $h$ is of order $n^{-1/3}$.

The true values being estimated are $P(x_0,y)=0.673043, 0.884243$, respectively for $h=0.625, 1.80$.

Figure~\ref{fig:nbc.arma} shows the mean squared error of the estimator $NBC$ of $P(x_0,y)$, as a function of $(b,\ell)$, while Figure~\ref{fig:uns.arma} shows the corresponding dependence for the estimator $UNS$, for bandwidth $h=0.625$. Note that the latter estimator is, in this example, much more accurate than $NBC$, as indicated by asymptotic theory, but, as noted, requires specification of the tuning constant $k_1$, which has been optimised over, using a grid search, for each $(b,\ell)$ combination shown in construction of the figure. The theory presented in Section~\ref{sec:theory} suggests that for optimality of the bootstrap estimator $NBC$ the block bootstrap sample size $b\ell$ should have smaller order than the sample size $n$. This conclusion incorporates the case $b=1$ seen to give smallest mean squared error in this example. For the bandwidth $h=0.625$ considered here, where the bias of the density estimator is small, the estimator $EBC$, optimised over the two tuning constants $k_1$ and $c_2$, yields accuracy figures indistinguishable from those of the simpler estimator $UNS$. Again in accordance with theory, for a given block length $\ell$, accuracy of the estimators $UNS$ and $EBC$ is improved by increasing the number of blocks $b$: overall optimality in terms of mean squared error of estimation is achieved by $(b,\ell)=(33,3)$. Increasing the value of the bandwidth $h$ to its optimal value for estimation of $f(x_0)$, $h=1.80$, allows for substantial reduction in error through explicit bias correction. Table~\ref{table:mse.n100} provides illustrative comparisons, for several combinations of $(b,\ell)$. In each case, the estimators $UNS$ and $EBC$ have been optimised over their respective tuning constants. Accuracy is seen to be very sensitive to values of these constants, as we will discuss further below. The estimator $NBC$ is quite inaccurate relative to the estimators $UNS$ and $EBC$, at least when the tuning constants required by the latter are fixed at close to optimal values.

\begin{figure}[ht]
\includegraphics[width=6in ] {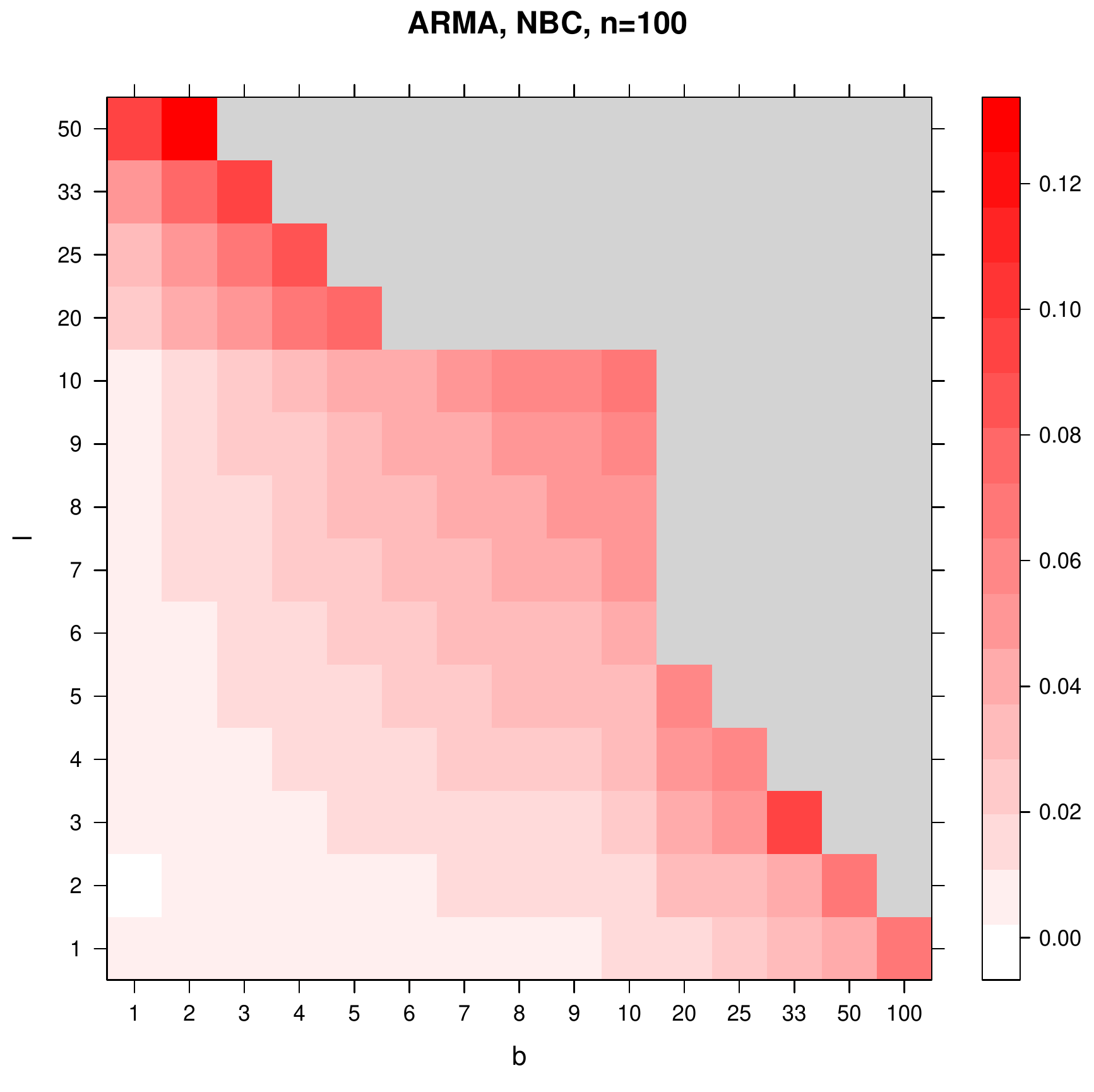}
\caption{\label{fig:nbc.arma} MSE of estimator $NBC$ for $ARMA$ example, $n=100$, as a function of $(b,\ell)$.}
\end{figure}
\begin{figure}[ht]
\includegraphics[width=6in ] {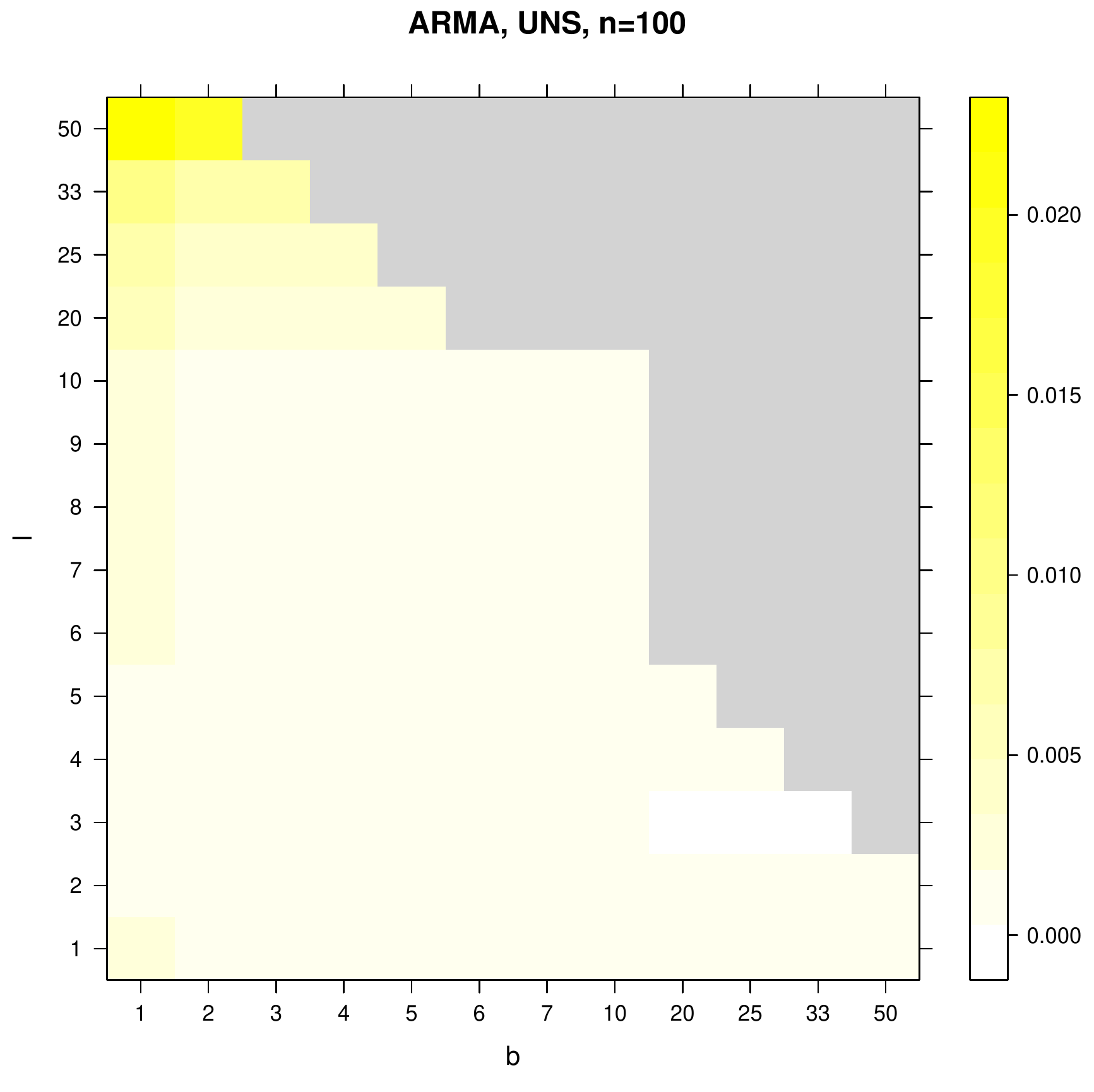}
\caption{\label{fig:uns.arma} MSE of estimator $UNS$ for $ARMA$ example, $n=100$, as a function of $(b,\ell)$.}
\end{figure}

\begin{table}[ht]
\renewcommand\thetable{1}
\caption{\label{table:mse.n100} MSEs of estimators $NBC$, $UNS$ and $EBC$ for $ARMA$ example, $n=100$, bandwidths $h=0.625, 1.80$.  }
\begin{center}
\begin{tabular}{c|rr|rrr|}
$(b,\ell)$ & \multicolumn{2}{c|}{$h=0.625$}&\multicolumn{3}{c|}{$h=1.80$} \\
&&&&& \\
&$NBC$& $UNS$/$EBC$ &$NBC$& $UNS$& $EBC$\\
$(1,2)$ &0.01136&0.00113&0.03593&0.00991&0.00805 \\
$(1,5)$&0.00827&0.00166& 0.03083&0.00808&0.00603 \\
$(1,10)$&0.00991&0.00294&0.02742&0.00796&0.00639 \\
$(5,2)$&0.00804&0.00046& 0.01767&0.00261&0.00133 \\
$(5,5)$&0.01576&0.00040&0.02357&0.00340&0.00241 \\
$(5,10)$&0.02811&0.00093&0.03385& 0.00408&0.00329 \\
$(10,2)$&0.01148&0.00040&0.01986& 0.00201&0.00115 \\
$(10,5)$ &0.02460&0.00038&0.03117& 0.00295&0.00222 \\
$(10,10)$&0.04703&0.00093& 0.04985&0.00373&0.00310
\end{tabular}
\end{center}
\end{table}

\begin{table}[ht]
\caption{\label{table:mse.n200} MSEs of estimators $NBC$, $UNS$ and $EBC$ for $ARMA$ example, $n=200$, bandwidths $h=0.82, 0.93$. }
\begin{center}
\begin{tabular}{c|rrr|rrr|}
$(b,\ell)$ & \multicolumn{3}{c|}{$h=0.82$}&\multicolumn{3}{c|}{$h=0.93$} \\
&&&&&& \\
&$NBC$& $UNS$& $EBC$ &$NBC$& $UNS$& $EBC$\\
&&&& &&\\
$(1,10)$ &0.00592&0.00171&0.00140&0.00639&0.00522& 0.00143\\
$(5,5)$ &0.00712&0.00082&0.00016 &0.00776&0.00529&0.00019\\
$(5,10)$&0.01384&0.00078&0.00040 &0.01492&0.00454&0.00028\\
$(10,5)$& 0.01390&0.00078&0.00014&0.01504&0.00517& 0.00017\\
$(20,5)$&0.02634&0.00073&0.00014&0.02848&0.00499&0.00018 \\
$(20,10)$&0.05033&0.00074&0.00040&0.05302&0.00434&0.00023 \\
$(40,5)$ &0.04994&0.00070&0.00015&0.05250&0.00488&0.00015 \\
$(50,4)$ &0.05038&0.00086&0.00010&0.05363&0.00546&0.00014
\end{tabular}
\end{center}
\end{table}

We consider now the larger sample size $n=200$, and estimation of $P(3.0,0.1)$. Two bandwidths are considered, $h=0.93$, for which $P(3.0,0.1)=0.632329$ and $h=0.82$,
for which $P(3.0,0.1)=0.680093$. The former bandwidth $h$ is that for which the normal approximation is most accurate, while the latter corresponds to the bandwidth which minimises the mean squared error of the density estimator: note that for this sample size $n=200$, this is smaller than the $h$ optimal for the normal approximation, a reversal of the asymptotic ordering and that seen for the $n=100$ case above.

Mean squared errors of the three estimators for several $(b, \ell)$ combinations are given in Table~\ref{table:mse.n200}. As before, figures for $UNS$ are those optimised over the tuning constant $k_1$, while those for $EBC$ are optimised over the two tuning constants $k_1$ and $c_2$, with the estimator constructed with $k_2=c_2n^{-1/9}$. 

In this context, bias of the kernel estimator $\hat f(x_0)$  is 0.0036446 at $h=0.82$, and 0.0047002 at $h=0.93$, and the figures confirm the advantage of the explicit bias correction estimator $EBC$ over the undersmoothing estimator $UNS$. Mean squared error figures for the estimator $NBC$ are again noticeably poor.

The results in Table~\ref{table:mse.n200} also suggest that if one is concerned about not so much the accuracy of the density estimator as the accuracy of the inference statement (e.g. confidence level, significance level) made out of that estimator, then a smaller bandwidth $h$ might be preferable to a bigger one. A downside, however, could be a drop in efficiency (leading to, e.g. a long confidence interval).

Figure~\ref{fig:uns.arma.n200} shows the acute dependence of the mean squared error of the estimator $UNS$ on the tuning constant $k_1$, for the case $(b, \ell)=(50,4), h=0.82$, and is typical of other cases studied. Effective practical implementation of the estimator will require efficient methods for empirical choice of $k_1$, as well as determination of the number of blocks $b$ and their length $\ell$ used in the construction. Figure~\ref{fig:ebc.arma.n200} shows, for the same case $(b,\ell)=(50,4), h=0.82$ dependence of the estimator $EBC$ on the two tuning constants $k_1, c_2$. Similar conclusions may be drawn about the need for effective setting of these values.

One further practical point indicated by the theory is worth mention. In principle, the theoretically optimal orders of the tuning constants $(k_1, k_2)$ for a fixed combination $(b, \ell)$ are only functions of $(n, b, \ell)$ and not the bandwidth $h$. Empirical investigations show, however, that in a finite sample situation, such as that illustrated in Figure~\ref{fig:ebc.arma.n200}, the constants $(k_1, k_2)$ which minimise the mean squared error of estimation of $P(x_0, y)$ may depend quite considerably on the bandwidth $h$ specified, leading to the conclusion that optimal tuning of the estimators $UNS$ and $EBC$ will, in practice, require also consideration of this value.

\begin{figure}
\includegraphics[width=6in ] {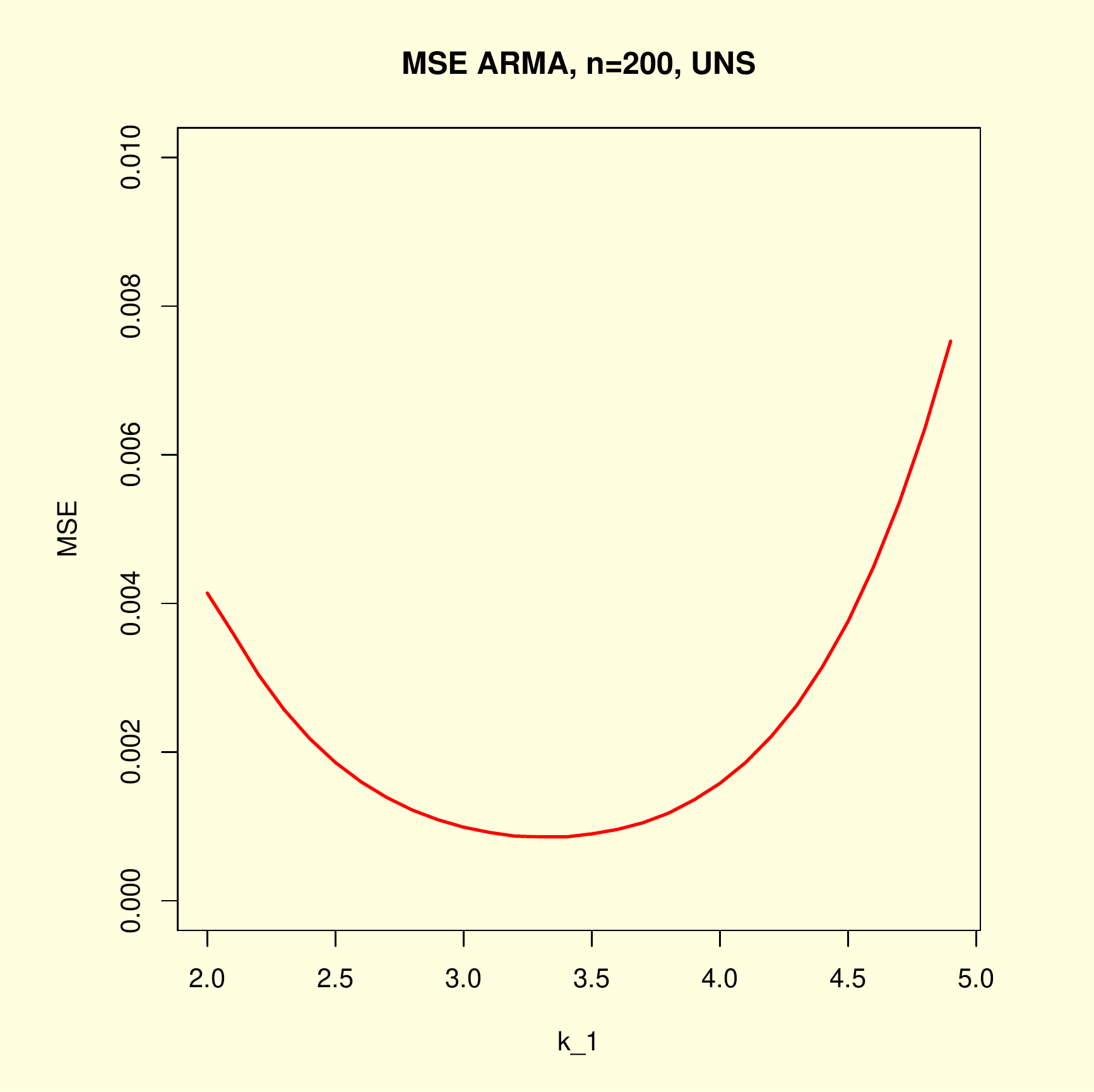}
\caption{\label{fig:uns.arma.n200} MSE of estimator $UNS$ for $ARMA$ example, $n=200$, as a function of tuning constant $k_1$, for $(b, \ell)=(50,4)$.}
\end{figure}
\begin{figure}
\includegraphics[width=6in ] {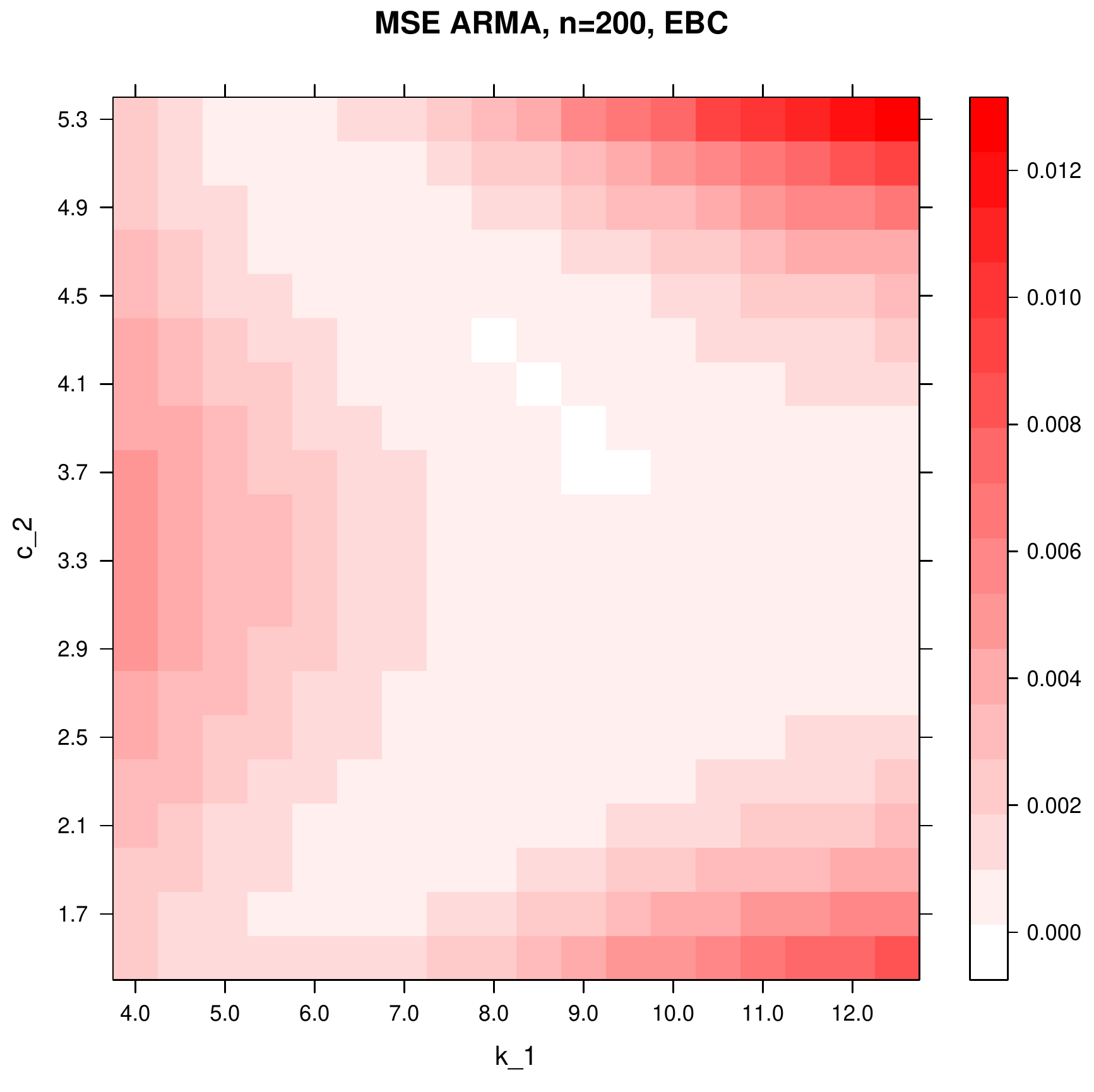}
\caption{\label{fig:ebc.arma.n200}
MSE of estimator $EBC$ for $ARMA$ example, $n=200$, as a function of tuning constants $k_1, c_2$, for $(b, \ell)=(50,4)$.}
\end{figure}

\section{Conclusions}
\label{sec:conclusion}
Many applications in statistics involve both dependent data and nonparametric kernel-based estimators, which are non-smooth (i.e., non-regular) functionals. Examples in time series analysis include density estimation, conditional mean estimation, prediction and local regression \citep{HLC:1997}.  This combination of dependence and a non-smooth statistical functional comprises a setting for which nonparametric bootstrap methodology is still poorly understood. It is important to know how to optimize the bootstrap procedure to achieve the fastest possible convergence rate. The development of such an optimality theory is relevant for scalability and computational efficiency of block bootstrap procedures, and also for consistent estimation of bias, optimal bandwidth selection, and coverage accuracy of block bootstrap confidence intervals.

In this paper, we have presented a general theory of optimality of the block bootstrap for estimation of the sampling distribution of a kernel density estimator with dependent observations arising from a strong mixing process. Optimality means achieving the fastest possible convergence rate of the block bootstrap estimator, which corresponds to minimizing the order of magnitude of estimation error of the block bootstrap estimator of the sampling distribution of the kernel density estimator. We have described
a unified framework for a theoretical study of a rich class of bootstrap methods for this problem which includes as special cases subsampling, K\"{unsch}'s moving block bootstrap, Hall's under-smoothing, as well as approaches incorporating no or explicit bias correction. Accuracy under a broad spectrum of choices of the bandwidth $h$, which include as an important special case the MSE-optimal choice, as well as other under-smoothed choices, has been analysed. Under each choice of $h$, we have detailed the optimal tuning parameters and compared optimal performances between the main subclasses (EBC, NBC, UNS) of the bootstrap methods.

Our theoretical results in this setting are the first of their kind. Previous work concerning block bootstrap methods for smooth functionals prescribed different fixed rules for the number of blocks $b$, whereas the block length $\ell$ was chosen to achieve the fastest convergence rate. By contrast, we also allow the number of blocks $b$ to be chosen optimally; doing so allows for a faster convergence rate quite generally. Examples of conventional fixed-rule $b$ regimes are the subsampling bootstrap for which $b=1$ and the moving block bootstrap for which $b=\lfloor n/\ell \rfloor$.
Our results show that the optimal convergence rate is achieved by simultaneously choosing the values of the pair $(b,\ell)$, and explain when the optimal choice of $b$ is $b=1$ (subsampling bootstrap), or $b=\lfloor n/\ell \rfloor$, or some value in-between these two extremes. Since for any intermediate value of $b$, this block bootstrap does not correspond to either of these two extremes, but can be interpreted as a compromise between them, we have called this new method the hybrid block bootstrap. The optimal choice of the pair $(b,\ell)$ depends on the bandwidth for the kernel density estimator, the rate of the decay of the mixing coefficients of the underlying stochastic process, the type or absence of studentization, and also whether or not an explicit bias correction is used. Compared to the sample quantile problem, the optimality theory we have established for the kernel density problem implies very different schemes for selecting $(b,\ell)$. The results for these two non-smooth functionals are dramatically different, and whether one performs a separate bias correction has enormous impact on the best convergence rates. Generally, standard MBB (with $n \approx b\ell$) can be tuned to work optimally, but subsampling cannot, if one does a separate bias correction. Without separate bias correction, a hybrid between standard MBB and subsampling is more favourable.

Recently, \citet{GLN:2015,GLN:2018} have proposed a new smoothed block bootstrap method showing promising results for the quantile and quantile regression problems. This method does not fall as a special case within the framework of our hybrid block bootstrap approach. These authors establish only consistency for several statistical functionals. A study of rates of convergence for the density estimation problem and comparison with the hybrid schemes considered here seems worthwhile future work. 


\appendix

\section{Proofs of Main Results and Auxiliary Lemmas}

In what follows we denote by $C^*$ a generic positive constant independent of $n$, and
by $\{L_n\}$ a generic positive, slowly varying, sequence converging to 0.
Define, for $h>0$ and $d\ge 1$, $\kappa_{n,d,h}$ to be the $d$-th cumulant of
$\hat{f}_h(x_0)$, and
\begin{align*}
&\mu_{n,d,h}= \\
&\expn\bigg[\left\{\dfrac{1}{h}K\left(\dfrac{X_0-x_0}{h}\right)-
\kappa_{n,1,h}\right\}
\bigg(\sum_{|t|\le n-1}\left\{\dfrac{1}{h}K\left(\dfrac{X_t-x_0}{h}\right)-
\kappa_{n,1,h}\right\}\bigg)^{d-1}\bigg].
\end{align*}
Note that $\mu_{n,1,h}\equiv 0$.

\begin{lem}
\label{lem:cum-mom}
Suppose that $\alpha(t)=O(t^{-\beta})$ as $t\rightarrow\infty$, for some $\beta>2$, and that
$nh\rightarrow\infty$ and $h=o(n^{-\epsilon})$ for some $\epsilon\in(0,1)$. Then, for any arbitrarily small $\delta>0$,
\begin{equation}
\kappa_{n,d,h}=
\begin{cases}
f(x_0)+h^{2}f''(x_0)\mu_2/2+O(h^4),& d=1,\\
(nh)^{-1}f(x_0)\nu_2+O\big(n^{-1}h^{-2/\beta}\big),& d=2,\\
O\big\{(nh)^{-d+1}+n^{-\beta+\delta}h^{-d}+n^{-d+1+\delta}
h^{-d\gamma_0(\beta,d)}\big\},&d\ge 3;
\end{cases}
\label{nonbt.cum}
\end{equation}
and
\begin{equation}
\mu_{n,d,h}=
\begin{cases}
O(h^{-1}),&d=2,\\
O\big(h^{-2}+n^{2-\beta}h^{-3}L_n^{-1}+n^\delta h^{-3\gamma_0(\beta,3)}\big),&d=3,\\
O\big(nh^{-2}+n^{3-\beta+\delta}h^{-4}+n^{\delta}
h^{-4\gamma_0(\beta,4)}\big),&d=4,\\
O\big(n^{d-\beta-1+\delta}h^{-d}+n^{(d-2)/2+\delta} h^{-d\max\{\gamma_0(\beta,d),1/2\}}\big),&d\ge 5.
\end{cases}
\label{nonbt.mom}
\end{equation}
\end{lem}

\begin{proof}[Proof of Lemma~\ref{lem:cum-mom}]
The case $d=1$ for (\ref{nonbt.cum}) follows readily by noting
\[ \kappa_{n,1,h}=\dfrac{1}{h}\int K\left(\dfrac{y-x_0}{h}\right)f(y)\,dy
=\int K(u)f(x_0+hu)\,du\]
and Taylor expansion of $f$ about $x_0$.

Consider next the case $d=2$. Setting $m_n=[h^{-2/\beta}]$, the integer part of $h^{-2/\beta}$, we have
\begin{align*}
\kappa_{n,2,h}&=n^{-1}h^{-2}{\rm Var}\big(K\left((X_0-x_0)/h\right)\big)\\
& +(nh)^{-2}\hspace{-2ex}
\sum_{1\le|t|\le m_n-1}\hspace{-2ex}(n-|t|)\text{Cov}\big(K\left((X_0-x_0)/h\right),K\left((X_t-x_0)/h\right)\big)\\
& +(nh)^{-2}\hspace{-2ex}
\sum_{m_n\le|t|\le n-1}\hspace{-2ex}(n-|t|)\text{Cov}\big(K\left((X_0-x_0)/h\right),K\left((X_t-x_0)/h\right)\big)\\
&=(nh)^{-1}f(x_0)\nu_2\big\{1+O(h^2)\big\}+O\big(n^{-1}m_n+n^{-1}h^{-2}m_n^{-\beta+1}\big)\\
&=(nh)^{-1}f(x_0)\nu_2+O\big(n^{-1}h^{-2/\beta}\big)
\end{align*}
and that
\begin{align*}
n^{-1}\mu_{n,2,h}&=\kappa_{n,2,h}+(nh)^{-2}\hspace{-1ex}
\sum_{|t|\le n-1}\hspace{-1ex}|t|\text{Cov}\big(K\left((X_0-x_0)/h\right),K\left((X_t-x_0)/h\right)\big)\\
&= O\big\{(nh)^{-1}\big\},
\end{align*}
which proves (\ref{nonbt.cum}) and (\ref{nonbt.mom}) for $d=2$.

Let $\{n_s\}_{0\le s\le S}$ be an arbitrary sequence of positive integers  satisfying $n=n_0$, $n_S\rightarrow\infty$ and $n_{s+1}=o(n_s)$
for $s=0,1,\ldots,S-1$.
Invoking mixing properties of $\{Y_t\}$, we have, for $d\ge 3$,
\[n_s^{d-1}\kappa_{n_s,d,h}=O\left(n_{s+1}^{d-1}\kappa_{n_{s+1},d,h}+h^{-d}n_s^{d-2}n_{s+1}^{-\beta+1}\right).\]
Suppose that $n_r^{d-1}\kappa_{n_r,d,h}=O\big(n_r^{\alpha_r}h^{-d\gamma_r}+n^{d-\beta-1}h^{-d}L_n^{-1}+h^{1-d}\big)$
for some $(\alpha_r,\gamma_r)$ and for $r=s+1,\ldots,S$.
Then we have, by setting 
\[
n_{s+1}\propto\min\left\{n_s^{\frac{d-2}{\alpha_{s+1}+\beta-1}}h^{\frac{d(\gamma_{s+1}-1)}{\alpha_{s+1}+\beta-1}},
n_sL_n\right\},
\]
that
\[ n_s^{d-1}\kappa_{n_s,d,h}=O\left(n_s^{\frac{(d-2)\alpha_{s+1}}{\alpha_{s+1}+\beta-1}}h^{-d\left\{
\frac{\gamma_{s+1}(\beta-1)+\alpha_{s+1}}{\alpha_{s+1}+\beta-1}\right\}}+n^{d-\beta-1}h^{-d}L_n^{-1}+h^{1-d}\right).\]
It follows by induction that
\[
n_s^{d-1}\kappa_{n_s,d,h}=O\big(n_s^{\alpha_s}h^{-d\gamma_s}+n^{d-\beta-1}h^{-d}L_n^{-1}+h^{1-d}\big)
\]
for all $s=0,1,\ldots,S$, where
$(\alpha_s,\gamma_s)$ satisfy the recursive relations
\[ \alpha_s=\frac{(d-2)\alpha_{s+1}}{\alpha_{s+1}+\beta-1}\qquad\text{and}\qquad
\gamma_s=\frac{\gamma_{s+1}(\beta-1)+\alpha_{s+1}}{\alpha_{s+1}+\beta-1},
\]
for $s=0,1,\ldots,S-1$.
Noting the general fact that
$n_S^{d-1}\kappa_{n_S,d,h}=O\big(\mu_{n_S,d,h}\big)=O\big(n_S^{d-1}+h^{1-d}\big)$, we may set
$(\alpha_S,\gamma_S)=(d-1,0)$ and obtain from the recursive relations that
\[ \alpha_0=\max\{d-\beta-1,0\}+\delta_1
\;\;\text{and}\;\;\gamma_0=\gamma_0(\beta,d)-\delta_2, \]
for some $\delta_1,\delta_2\in(0,\delta)$, provided that $S$ is sufficiently large. It follows that
\begin{align}
n^{d-1}\kappa_{n,d,h}&=O\big(n^{\max\{d-\beta-1,0\}+\delta_1}h^{-d\gamma_0(\beta,d)+d\delta_2}+n^{d-\beta-1}h^{-d}L_n^{-1}+h^{1-d}\big)
\label{pf:lem1.1}\\
&=O\big(n^{\delta}h^{-d\gamma_0(\beta,d)}+n^{d-\beta-1+\delta}h^{-d}+h^{1-d}\big),\nonumber
\end{align}
using the fact $\gamma_0(\beta,d)=1$ for $\beta\le d-1$. This proves (\ref{nonbt.cum}) for $d\ge 3$.
Furthermore, the case $d=3$ of (\ref{nonbt.mom}) follows from (\ref{pf:lem1.1}) by noting that
$\mu_{n,3,h}=O\big(n^2\kappa_{n,3,h}\big)$.

Finally, for the case $d\ge 4$, note that
\begin{equation}
\label{pf:lem1.2}
\mu_{n,d,h}=O\Big(n^{d-1}\kappa_{n,d,h}+n\sum_{d_1=2}^{d-2}\big|\mu_{n,d_1,h}\mu_{n,d-d_1,h}\big|\Big).
\end{equation}
The special case $d=4$ follows immediately by noting (\ref{nonbt.cum}) and that $\mu_{n,2,h}=O(h^{-1})$.
Assuming that (\ref{nonbt.mom}) holds for any $d<D$ with $D\ge 5$, it then follows from (\ref{pf:lem1.2}) that
(\ref{nonbt.mom}) holds also for $d=D$, using the fact that
\begin{multline*}
n\times n^{D-d_1-\beta-1+\delta}h^{-D+d_1}\Big\{
n^{d_1-\beta-1+\delta}h^{-d_1}+n^{(d_1-2)/2+\delta} h^{-d_1\max\{\gamma_0(\beta,d_1),1/2\}}\Big\}
\\=O\big(n^{D-\beta-1+2\delta}h^{-D}\big)
\end{multline*}
and
\begin{multline*}
n\times n^{(d_1-2)/2+\delta} h^{-d_1\max\{\gamma_0(\beta,d_1),1/2\}}\times
n^{(D-d_1-2)/2+\delta} h^{-(D-d_1)\max\{\gamma_0(\beta,D-d_1),1/2\}}\\
=O\big(n^{(D-2)/2+2\delta} h^{-D\max\{\gamma_0(\beta,D),1/2\}}\big).
\end{multline*}
The proof of (\ref{nonbt.mom}) is thus completed by induction, noting that the order displayed in (\ref{nonbt.mom})
of $\mu_{n,d,h}$ for $d\ge 5$ holds trivially for the cases $d=2,3,4$ as well.
\end{proof}

\begin{proof}[Proof of Theorem~\ref{thm:nonbt}]
Write $\,\U_t=h^{-1}K\big((X_{t}-x_0)/h\big)-\kappa_{n,1,h}$.
We follow the small-block and large-block arguments \citep[Section~6.6.2]{FanYao:2003} by defining
\begin{gather*}
V_L=\sqrt{\dfrac{h}{n}}\sum_{i=1}^{M}\sum_{t=1}^\lambda \U_{t+(i-1)(\lambda+\xi)},\qquad
V_S=\sqrt{\dfrac{h}{n}}\sum_{i=1}^{M}\sum_{t=1}^\xi \U_{t+\lambda+
(i-1)(\lambda+\xi)},\\
V_R=(nh)^{1/2}\big\{\hat{f}_h(x_0)-
\kappa_{n,1,h}\big\}-V_L-V_S,
\end{gather*}
where $\lambda,\xi$ are positive integers satisfying $\xi\rightarrow\infty$, $\xi+h^{-1}=o(\lambda)$ and $\lambda=o(n)$, and
$M=\big[n/(\lambda+\xi)\big]$. 
Note by (\ref{nonbt.mom}) that
\begin{equation}
\label{pf:thm1.1}
\text{Var}(V_R)=O\big\{n^{-1}h(n-M\lambda-M\xi)\mu_{n,2,h}\big\}
=O(n^{-1}\lambda),
\end{equation}
and
\begin{align}
\text{Var}(V_S)&= O\Big(
\xi Mn^{-1}h\sum_{|t|\le\xi}\sum_{|i|\le M}\expn\left[\U_0\U_{t+i(\lambda+\xi)}\right]\Big)\nonumber\\
&=O\bigg\{
\xi Mn^{-1}h\sum_{|t|\le\xi}\Big(\expn\left[\U_0\U_{t}\right]+\sum_{1\le|i|\le M}
h^{-2}(i\lambda)^{-\beta}\Big)\bigg\}\nonumber\\
&=O\Big\{
\xi\lambda^{-1}\big(1+\xi\lambda^{-\beta}h^{-1}\big)\Big\}.
\label{pf:thm1.2}
\end{align}
Writing $\kappa_d(V_L)$ for the $d$-th cumulant of $V_L$ and noting (\ref{nonbt.cum}), we have
\begin{align}
\kappa_2(V_L)&=M\lambda^2(h/n)\kappa_{\lambda,2,h}
+O\big(
M\xi^{-\beta}\big)\nonumber\\
&=
f(x_0)\nu_2+O\big(\xi\lambda^{-1}+\lambda n^{-1}+h^{1-2/\beta}+\xi^{-\beta}\lambda^{-1}n\big)
\label{pf:thm1.3}
\end{align}
and, for $d\ge 3$ and arbitrarily small $\delta>0$,
\begin{align}
\kappa_d(V_L)&=M\lambda^d(h/n)^{d/2}\kappa_{\lambda,d,h}
+O\big(
M\xi^{-\beta}\big)\nonumber\\
&=
O\big\{(nh)^{1-d/2}
+\lambda^{d-\beta-1+\delta}n^{1-d/2}h^{-d/2}\nonumber\\&\qquad +\,\lambda^{\delta}n^{1-d/2}
h^{d/2-d\gamma_0(\beta,d)}+\xi^{-\beta}\lambda^{-1}n\big\}.
\label{pf:thm1.4.1}
\end{align}
Under the additional assumption that $\lambda^2=O(nh)$, the order terms in (\ref{pf:thm1.4.1}) can be bounded further to yield
\begin{multline}
\kappa_d(V_L)=O\big\{(nh)^{-1/2}
+\lambda^{2-\beta+\delta}n^{-1/2}h^{-3/2}\\+
\lambda^{\delta}n^{-1/2}h^{-(3\beta+3)/(2\beta-2)+g_0(\beta)}+\xi^{-\beta}\lambda^{-1}n\big\}.
\label{pf:thm1.4}
\end{multline}
It follows from (\ref{pf:thm1.1}) to (\ref{pf:thm1.4}) that
\begin{eqnarray}
\lefteqn{\prob\Big((nh)^{1/2}\big\{\hat{f}_h(x_0)-f(x_0)\big\}\le y\Big)}\nonumber\\
&=& \prob\Big(V_L+(nh)^{1/2}\big\{\kappa_{n,1,h}-f(x_0)\big\}\le y\Big)\nonumber\\
&&\qquad+\,O\Big\{n^{-1/2}\lambda^{1/2}+\xi^{1/2}\lambda^{-1/2}\big(1+\xi^{1/2}\lambda^{-\beta/2}h^{-1/2}\big)
\Big\}\nonumber
\\
&=&
\Phi\left(\dfrac{y-(nh)^{1/2}\big\{\kappa_{n,1,h}-f(x_0)\big\}}{\sqrt{f(x_0)\nu_2}}\right)\nonumber\\
&&+\,O\Big\{\xi^{1/2}\lambda^{-1/2}+\xi^{-\beta}\lambda^{-1}n+h^{1-2/\beta}+
n^{-1/2+\delta}
h^{-(3\beta+3)/(2\beta-2)+g_0(\beta)}\Big\}.
\label{pf:thm1.5}
\end{eqnarray}
Theorem~\ref{thm:nonbt} then follows by setting 
$\lambda=\big[(nh)^{1/2}\big]$ and $\xi=\big[(n^3h^{-1})^{1/(4\beta+2)}\big]$ in (\ref{pf:thm1.5}), and noting
the expansion for $\kappa_{n,1,h}$ given in (\ref{nonbt.cum}). 
\end{proof}

\begin{lem}
\label{lem:cum-mom.expo}
Suppose that $\alpha(t)=O(e^{-Ct})$ as $t\rightarrow\infty$, for some $C>0$, and that $nh\rightarrow\infty$ 
and $h=o(n^{-\epsilon})$ for some $\epsilon\in(0,1)$. Then we have
\begin{equation}
\kappa_{n,d,h}=
\begin{cases}
f(x_0)+h^{2}f''(x_0)\mu_2/2+O(h^4),& d=1,\\
(nh)^{-1}f(x_0)\nu_2+O\big(n^{-1}\log n\big),& d=2,\\
O\big\{(nh)^{-d+1}\big\},&d\ge 3;
\end{cases}
\label{nonbt.cum.expo}
\end{equation}
and
\begin{equation}
\mu_{n,d,h}=
\begin{cases}
O(h^{-d+1}),&d=2,3,\\
O\big(n^{(d-2)/2} h^{-d/2}\big),&d\ge 4.
\end{cases}
\label{nonbt.mom.expo}
\end{equation}
\end{lem}

\begin{proof}[Proof of Lemma~\ref{lem:cum-mom.expo}]
The expansion for $\kappa_{n,1,h}$ can be proved in the same way as that given in the proof of Lemma~\ref{lem:cum-mom}.

Consider next the case $d=2$. Setting $m_n=\big[2(C\epsilon)^{-1}\log n\big]$, we have
\begin{align*}
\kappa_{n,2,h}&=n^{-1}h^{-2}{\rm Var}\big(K\left((X_0-x_0)/h\right)\big)\\
& +(nh)^{-2}\hspace{-2ex}
\sum_{1\le|t|\le m_n-1}\hspace{-2ex}(n-|t|)\text{Cov}\big(K\left((X_0-x_0)/h\right),K\left((X_t-x_0)/h\right)\big)\\
& +(nh)^{-2}\hspace{-2ex}
\sum_{m_n\le|t|\le n-1}\hspace{-2ex}(n-|t|)\text{Cov}\big(K\left((X_0-x_0)/h\right),K\left((X_t-x_0)/h\right)\big)\\
&=(nh)^{-1}f(x_0)\nu_2\big\{1+O(h^2)\big\}+O\big(n^{-1}m_n+n^{-1}h^{-2}e^{-Cm_n}\big)\\
&=(nh)^{-1}f(x_0)\nu_2+O\big(n^{-1}\log n\big),
\end{align*}
and that
\begin{align*}
\mu_{n,2,h}&=n\kappa_{n,2,h}+n^{-1}h^{-2}\hspace{-1ex}
\sum_{|t|\le n-1}\hspace{-1ex}|t|\text{Cov}\big(K\left((X_0-x_0)/h\right),K\left((X_t-x_0)/h\right)\big)\\
&= O\big(h^{-1}+n^{-1}m_n^2+n^{-1}h^{-2}m_ne^{-Cm_n}\big)=O(h^{-1}),
\end{align*}
which proves (\ref{nonbt.cum.expo}) and (\ref{nonbt.mom.expo}) for $d=2$.

Consider now the case $d\ge 3$. Setting $k_n=\big[n^\epsilon\big]$, we have
\[
n^{d-1}\kappa_{n,d,h}=O\left(k_n^{d-1}\kappa_{k_n,d,h}+h^{-d}n^{d-2}e^{-Ck_n}\right)
=O\big(k_n^{d-1}+h^{1-d}\big)=O(h^{1-d}),\]
so that (\ref{nonbt.cum.expo}) holds for $d\ge 3$.
That (\ref{nonbt.mom.expo}) holds for $d=3$ follows by noting that
$\mu_{n,3,h}=O\big(n^2\kappa_{n,3,h}\big)$.

It is clear that the order displayed in (\ref{nonbt.mom.expo})
of $\mu_{n,d,h}$ for $d\ge 4$ holds for the cases $d=2,3$. Noting (\ref{pf:lem1.2}) and (\ref{nonbt.cum.expo}), we have,
for $d\ge 4$, that
\[
\mu_{n,d,h}=O\Big(h^{-d+1}+n\sum_{d_1=2}^{d-2}\big|\mu_{n,d_1,h}\mu_{n,d-d_1,h}\big|\Big),
\]
so that (\ref{nonbt.mom.expo}) holds for $d\ge 4$ by induction.
\end{proof}

\begin{proof}[Proof of Theorem~\ref{thm:nonbt.expo}]
Note first that by (\ref{nonbt.cum.expo}), 
\[(nh)^{d/2}\kappa_{n,d,h}=
\begin{cases}
(nh)^{1/2}f(x_0)+n^{1/2}h^{5/2}f''(x_0)\mu_2/2+O\big(n^{1/2}h^{9/2}\big),&d=1,\\
f(x_0)\nu_2+O\big(h\log n\big),& d=2,\\
O\big\{(nh)^{-(d-2)/2}\big\}=O\big\{(nh)^{-1/2}\big\}, & d\ge 3.
\end{cases}
\]
It then follows that
\begin{eqnarray}
\lefteqn{\prob\Big((nh)^{1/2}\big\{\hat{f}_h(x_0)-f(x_0)\big\}\le y\Big)}\nonumber\\
&=&
\Phi\left(\dfrac{y-n^{1/2}h^{5/2}f''(x_0)\mu_2/2}{\sqrt{f(x_0)\nu_2}}\right)
+O\big\{h\log n+(nh)^{-1/2}+n^{1/2}h^{9/2}\big\},\nonumber
\end{eqnarray}
which implies Theorem~\ref{thm:nonbt.expo}, using the fact that $h=O(n^{-1/5})$.
\end{proof}

Define $\hat\kappa^{(d)}_{b,\ell,k}$ to be the $d$-th conditional cumulant of
$\hat{f}^*_{b,\ell,k}$ given $X_1,\ldots,X_n$.
\begin{lem}
\label{lem:bt.cum}
Assume that $\alpha(t)=O(t^{-\beta})$ as $t\rightarrow\infty$, for some $\beta>2$.
Suppose that $b\ell=O(n)$, $\ell=O(bk)$ and
$n^{-1}\ell+(\ell k)^{-1}+k+(nk_2)^{-1}+k_2=O(n^{-\epsilon})$,
for some arbitrarily small $\epsilon>0$. Then, for any arbitrarily small $\delta>0$,
\begin{align}
\hat\kappa^{(1)}_{b,\ell,k_2}&=f(x_0)+k_2^2f''(x_0)\mu_2/2+O_p\big\{k_2^4+(nk_2)^{-1/2}\big\},
\label{bt.cum1}\\[1ex]
\hat\kappa^{(2)}_{b,\ell,k}&=(b\ell k)^{-1}f(x_0)\nu_2+O_p\big(
b^{-1}\ell^{-1}k^{-2/\beta}\big),
\label{bt.cum2}\\[1ex]
\hat\kappa^{(d)}_{b,\ell,k}&=O_p\Big\{(b\ell k)^{1-d}+b^{1-d}\ell^{-\beta+\delta}k^{-d}
\label{bt.cumd}\\&\qquad+(b\ell)^{1-d+\delta}k^{-d\gamma_0(\beta,d)}
+n^{-1/2+\delta}b^{1-d}\ell^{(1-\beta)/2}k^{-d}\nonumber\\
&\qquad+n^{-1/2+\delta}b^{1-d}\ell^{(1-d)/2}k^{-d\max\{1/2,\gamma_0(\beta,2d)\}}
\Big\},\quad d\ge3.
\nonumber
\end{align}
\end{lem}

\begin{proof}[Proof of Lemma~\ref{lem:bt.cum}]

Writing $\V^{(d)}_{n,\ell,k}=(n-\ell+1)^{-1}\sum_{i=1}^{n-\ell+1}\big(U_{i,k,\ell}-\kappa_{\ell,1,k}\big)^d$, we have,
for $d\ge 1$,
\begin{align*}
\expn\big[\V^{(d)}_{n,\ell,k}\big]&=\expn\Big[\big(U_{1,k,\ell}-\kappa_{\ell,1,k}\big)^d
\Big]=O\big(\ell^{1-d}\mu_{\ell,d,k}\big), \\
n\text{Var}\big(\V^{(d)}_{n,\ell,k}\big)&=O\big(\ell^{1-\beta}k^{-2d}+\ell^{2-2d}\mu_{\ell,2d,k}\big).
\end{align*}
It then follows that
\begin{multline}
\label{pf:lem3.1}
\V^{(d)}_{n,\ell,k}=\expn\Big[\big(U_{1,k,\ell}-\kappa_{\ell,1,k}\big)^d
\Big]\\+O_p\Big(n^{-1/2}\ell^{(1-\beta)/2}k^{-d}
+n^{-1/2}\ell^{1-d}\big|\mu_{\ell,2d,k}\big|^{1/2}\Big).
\end{multline}
Note that
$\hat\kappa^{(1)}_{b,\ell,k_2}=(n-\ell+1)^{-1}\sum_{i=1}^{n-\ell+1}U_{i,k_2,\ell}=\kappa_{\ell,1,k_2}+\V^{(1)}_{n,\ell,k_2}$, so that
(\ref{bt.cum1}) follows by (\ref{nonbt.cum}) and (\ref{pf:lem3.1}).

Consider next the case $d=2$. Using Lemma~\ref{lem:cum-mom} and (\ref{pf:lem3.1}), we have
\begin{align*}
\hat\kappa^{(2)}_{b,\ell,k}&=b^{-1}\Big\{\V^{(2)}_{n,\ell,k}-\big(\V^{(1)}_{n,\ell,k}\big)^2\Big\}\\
&=b^{-1}\kappa_{\ell,2,k}+O_p\Big(n^{-1/2}b^{-1}\ell^{(1-\beta)/2}k^{-2}
+n^{-1/2}b^{-1}\ell^{-1}\big|\mu_{\ell,4,k}\big|^{1/2}+(nbk)^{-1}\Big)\\
&=(b\ell k)^{-1}f(x_0)\nu_2+O\big(b^{-1}\ell^{-1}k^{-2/\beta}\big)+O_p(\Delta_n),
\end{align*}
where 
\begin{align*}
\Delta_n&\equiv 
n^{-1/2+\delta}b^{-1}\ell^{(1-\beta)/2}k^{-2}+(nbk)^{-1}+n^{-1/2}b^{-1}\ell^{-1/2}k^{-1} \\
& \quad \quad +n^{-1/2+\delta}b^{-1}\ell^{-1}k^{-2\gamma_0(\beta,4)}\\
&=O_p\big(b^{-1}\ell^{-1}k^{-2/\beta}\big),
\end{align*}
which proves (\ref{bt.cum2}).

For $d\ge 3$, we have, using Lemma~\ref{lem:cum-mom} and (\ref{pf:lem3.1}) again, that
\begin{align*}
b^{d-1}&\hat\kappa^{(d)}_{b,\ell,k}=
\kappa_{\ell,d,k}+O_p\Big(n^{-1/2}\ell^{(1-\beta)/2}k^{-d}
+n^{-1/2}\ell^{1-d}\big|\mu_{\ell,2d,k}\big|^{1/2}\Big)\\
+&\, O_p\left\{
\sum_{d_1=2}^{d-2}\V^{(d_1)}_{n,\ell,k}\Big(n^{-1/2}\ell^{(1-\beta)/2}k^{-d+d_1}
+n^{-1/2}\ell^{1-d+d_1}\big|\mu_{\ell,2(d-d_1),k}\big|^{1/2}\Big)
\right\}\\
&= O_p\Big\{(\ell k)^{-d+1}+\ell^{-\beta+\delta}k^{-d}+\ell^{-d+1+\delta}
k^{-d\gamma_0(\beta,d)}+n^{-1/2+\delta}\ell^{(1-\beta)/2}k^{-d}\\
& \qquad\qquad+\,
n^{-1/2+\delta}\ell^{(1-d)/2} k^{-d\max\{\gamma_0(\beta,2d),1/2\}}\Big\},
\end{align*}
which proves (\ref{bt.cumd}).
\end{proof}

\begin{proof}[Proof of Theorem~\ref{thm:boot}]
Noting (\ref{nonbt.cum}) and (\ref{bt.cum1}), we have
\begin{multline}
\label{pf:thm3.1}
\expn\big[\hat{f}^*_{b,\ell,k_2}\big|X_1,\ldots,X_n\big]-\hat{f}_{k_3}(x_0)
=\hat\kappa^{(1)}_{b,\ell,k_2}-\kappa_{n,1,k_3}+O_p\big(\kappa_{n,2,k_3}^{1/2}\big)\\
= \big(k_2^2-k_3^2\big)f''(x_0)\mu_2/2
+O_p\big\{k_2^4+k_3^4+(nk_2)^{-1/2}+(nk_3)^{-1/2}\big\}.
\end{multline}
From (\ref{bt.cum2}) we deduce that
\begin{equation}
\text{Var}\big((b\ell k_1)^{1/2}\hat{f}^*_{b,\ell,k_1}\big)
=b\ell k_1\hat\kappa^{(2)}_{b,\ell,k_1}=f(x_0)\nu_2+O_p\big(
k_1^{1-2/\beta}\big).
\label{pf:thm3.2}
\end{equation}
For $d\ge 3$, we have, by (\ref{bt.cumd}), that the $d$-th cumulant of $(b\ell k_1)^{1/2}\hat{f}^*_{b,\ell,k_1}$ has the expression
\begin{multline}\label{pf:thm3.3}
(b\ell k_1)^{d/2}\hat\kappa^{(d)}_{b,\ell,k_1} 
=O_p\bigg\{(b\ell k_1)^{1-d/2}+n^{-1/2+\delta}b\ell^{(1-\beta)/2}\\
+\left(\dfrac{\ell}{bk_1}\right)^{d/2}
\Big[b\ell^{-\beta+\delta}
+(\ell k_1)^{-d}k_1^{d(2-\gamma_0(\beta,d))}(b\ell)^{1+\delta}\\
+(\ell k_1)^{-d/2}n^{-1/2+\delta}b\ell^{1/2}k_1^{(d/2)(2-\max\{2\gamma_0(\beta,2d)-1,0\})}\Big]
\bigg\} \\
=O_p\Big\{(b\ell k_1)^{-1/2}
+b^{-1/2}\ell^{-1/2+\delta}k_1^{-3/2+3g_1(\beta)} \\
+ n^{-1/2+\delta}b^{-1/2}\ell^{1/2}k_1^{(3/2)(g_2(\beta)-1)}
+n^{-1/2+\delta}b^{-1/2}\ell^{2-\beta/2}k_1^{-3/2}
\Big\}.
\end{multline}
Theorem~\ref{thm:boot} then follows by (\ref{pf:thm3.1}), (\ref{pf:thm3.2}), (\ref{pf:thm3.3}) and noting that
\[(b\ell k_1)^{-1/2}
+n^{-1/2+\delta}b^{-1/2}\ell^{2-\beta/2}k_1^{-3/2}=O_p\big(
k_1^{1-2/\beta}\big)\]
for sufficiently small $\delta>0$.
\end{proof}

\begin{lem}
\label{lem:bt.cum.expo}
Assume that $\alpha(t)=O(e^{-Ct})$ as $t\rightarrow\infty$, for some $C>0$.
Suppose that $b\ell=O(n)$ and
$n^{-1}\ell+(\ell k)^{-1}+k+(nk_2)^{-1}+k_2=o(1)$.
Then, for any arbitrarily small $\delta>0$,
\begin{align}
\hat\kappa^{(1)}_{b,\ell,k_2}&=f(x_0)+k_2^2f''(x_0)\mu_2/2+O_p\big\{k_2^4+(nk_2)^{-1/2}\big\},
\label{bt.cum1.expo}\\[1ex]
\hat\kappa^{(2)}_{b,\ell,k}&=(b\ell k)^{-1}f(x_0)\nu_2+O_p\big(
b^{-1}\ell^{-1}\log\ell+n^{-1/2}b^{-1}\ell^{-1/2}k^{-1}\big),
\label{bt.cum2.expo}\\[1ex]
\hat\kappa^{(d)}_{b,\ell,k}&=O_p\Big\{(b\ell k)^{1-d}
+n^{-1/2}b^{1-d}\ell^{(1-d)/2} k^{-d/2}\Big\},\qquad d\ge3.
\label{bt.cumd.expo}
\end{align}
\end{lem}

\begin{proof}[Proof of Lemma~\ref{lem:bt.cum.expo}]
Arguing as in the proof of Lemma~\ref{lem:bt.cum}, we have,
for $d\ge 1$, 
\[
\expn\big[\V^{(d)}_{n,\ell,k}\big]=O\big(\ell^{1-d}\mu_{\ell,d,k}\big)
\]
and that
\begin{equation}
\label{pf:lem4.1}
\begin{split}
n\text{Var}\big(\V^{(d)}_{n,\ell,k}\big)=O\big(e^{-C\ell}k^{-2d}+\ell^{2-2d}\mu_{\ell,2d,k}\big)=O\big(\ell^{2-2d}\mu_{\ell,2d,k}\big),\\
\V^{(d)}_{n,\ell,k}=\expn\Big[\big(U_{1,k,\ell}-\kappa_{\ell,1,k}\big)^d
\Big]+O_p\Big(n^{-1/2}\ell^{1-d}\big|\mu_{\ell,2d,k}\big|^{1/2}\Big).
\end{split}
\end{equation}
The proof of (\ref{bt.cum1.expo}) is exactly the same as that proving (\ref{bt.cum1}).
The expansion (\ref{bt.cum2.expo}) follows by Lemma~\ref{lem:cum-mom.expo}, (\ref{pf:lem4.1})
and the fact that $\hat\kappa^{(2)}_{b,\ell,k}=b^{-1}\big\{\V^{(2)}_{n,\ell,k}-\big(\V^{(1)}_{n,\ell,k}\big)^2\big\}$.

For $d\ge 3$, we have, using Lemma~\ref{lem:cum-mom.expo} and (\ref{pf:lem4.1}) again, that
\begin{align*}
b^{d-1}\hat\kappa^{(d)}_{b,\ell,k}&=
\kappa_{\ell,d,k}+O_p\Big(n^{-1/2}\ell^{1-d}\big|\mu_{\ell,2d,k}\big|^{1/2}\Big)\\
&+\,O_p\Big\{
n^{-1/2}\ell^{2-d}\sum_{d_1=2}^{d-2}\Big(\mu_{\ell,d_1,k}+n^{-1/2}\big|\mu_{\ell,2d_1,k}\big|^{1/2}
\Big)\big|\mu_{\ell,2(d-d_1),k}\big|^{1/2}
\Big\}\\
&= O_p\Big\{(\ell k)^{-d+1}+n^{-1/2}\ell^{(1-d)/2} k^{-d/2}\Big\},
\end{align*}
which proves (\ref{bt.cumd.expo}).
\end{proof}

\begin{proof}[Proof of Theorem~\ref{thm:boot.expo}]
In view of
(\ref{nonbt.cum.expo}) and (\ref{bt.cum1.expo}), (\ref{pf:thm3.1}) holds for 
\[
\expn\big[\hat{f}^*_{b,\ell,k_2}\big|X_1,\ldots,X_n\big]-\hat{f}_{k_3}(x_0)
\]
under the conditions of  Theorem~\ref{thm:boot.expo}.
From (\ref{bt.cum2.expo}) we deduce that
\begin{equation}
\text{Var}\big((b\ell k_1)^{1/2}\hat{f}^*_{b,\ell,k_1}\big)
=b\ell k_1\hat\kappa^{(2)}_{b,\ell,k_1}=f(x_0)\nu_2+O_p\big(
k_1\log\ell+n^{-1/2}\ell^{1/2}\big).
\label{pf:thm4.2}
\end{equation}
For $d\ge 3$, we have, by (\ref{bt.cumd.expo}), that the $d$-th cumulant of $(b\ell k_1)^{1/2}\hat{f}^*_{b,\ell,k_1}$ has the expression
\begin{align}
(b\ell k_1)^{d/2}\hat\kappa^{(d)}_{b,\ell,k_1}&=O_p\big\{
(b\ell k_1)^{1-d/2}
+n^{-1/2}b^{1-d/2}\ell^{1/2}\big\}\nonumber\\
&=O_p\big\{
(b\ell k_1)^{-1/2}
+n^{-1/2}b^{-1/2}\ell^{1/2}\big\}.
\label{pf:thm4.3}
\end{align}
Theorem~\ref{thm:boot.expo} then follows by (\ref{pf:thm3.1}), (\ref{pf:thm4.2}) and (\ref{pf:thm4.3}).
\end{proof}

\begin{proof}[Proof of Theorem~\ref{thm:boot1}]
Note first that Theorem~\ref{thm:boot} and (\ref{bias.cond}) together imply 
\[ \prob\big(\hat{T}^*_{b,\ell,k_1}\le y\big|X_1,\ldots,X_n\big)=
\Phi\left(\dfrac{y-n^{1/2}h^{5/2}f''(x_0)\mu_2/2}{\sqrt{f(x_0)\nu_2}}\right)+O_p(\Xi_n),\]
where $\Xi_n=k_1^{(\beta-2)/\beta}+b^{-1/2}\ell^{-1/2+\delta^*}k_1^{-3/2+3g_1(\beta)}+n^{-1/2+\delta^*}b^{-1/2}\ell^{1/2}k_1^{-3(1-g_2(\beta))/2}$,
for any arbitrarily small $\delta^*>0$. It suffices to show
\begin{equation}
\Xi_n=o\big(n^{-(\beta-1)(\beta-2)/(5\beta^2-5\beta-4)}\big).
\label{pf:thm5.0}
\end{equation}

Consider case (i) $\beta\le\beta_1$. If 
$b=O\big(n^{(2 -3 \beta +3 \beta g_1(\beta))/(4-5\beta + 6\beta g_1(\beta))}\big)$, then we have
$\ell\propto b^{1+\beta/(2 -3\beta + 3\beta g_1(\beta))}$ and
$k_1\propto n^{\delta'}\ell/b$, so that
\begin{align}
 \Xi_n&=O\big(n^{\delta'(\beta-2)/\beta}b^{(\beta-2)/(2 -3 \beta  +3 \beta g_1(\beta))}\big) \nonumber \\
&=O\big(n^{\delta'(\beta-2)/\beta+(b_{min}(\beta)+2\delta)(\beta-2)/(2 -3 \beta  +3 \beta g_1(\beta))}\big)\nonumber\\
&=o\big(n^{-(\beta-1)(\beta-2)/(5\beta^2-5\beta-4)-\delta(\beta-2)(2 + \beta + 3 \beta g_1(\beta))/
(6\beta^2(1-g_1(\beta))-4 \beta)}\big).
\label{pf:thm5.1}
\end{align}
If 
$bn^{-(2 -3 \beta +3 \beta g_1(\beta))/(4-5\beta + 6\beta g_1(\beta))}\rightarrow\infty$ and
$b=O\big(n^{1 + \beta/(4 - 5 \beta + 6 \beta g_1(\beta))}\big)$, then we have $\ell\propto n/b$ and $k_1\propto
 n^{-\beta/(5\beta-4-6\beta g_1(\beta))+\delta'}$, so that
\begin{equation}
\Xi_n=O\big(n^{-(\beta-2)/(5\beta-6\beta g_1(\beta)-4)+\delta'(\beta-2)/\beta}\big).
\label{pf:thm5.2}
\end{equation}
If 
$bn^{-1-\beta/(4 - 5 \beta + 6 \beta g_1(\beta))}\rightarrow\infty$, then we have $\ell\propto n/b$ and $k_1\propto n^{-1+\delta'}b$, so that
\begin{equation}
\Xi_n=O\big\{(n^{-1+\delta'}b)^{(\beta-2)/\beta}\big\}
=o\big(n^{\{b_{max}(\beta)-1-\delta/2\}(\beta-2)/\beta}\big).
\label{pf:thm5.3}
\end{equation}

Consider next case (ii) $\beta>\beta_1$. If $b=O(n^{2/3})$, then
$\ell\propto b^{1/2}$ and $k_1\propto n^{\delta'}b^{-1/2}$, so that
\begin{align}
\Xi_n&=O\big\{(b^{-1/2}n^{\delta'})^{(\beta-2)/\beta}+
n^{-1/2-3\delta'(1-g_2(\beta))/2}b^{(2 - 3 g_2(\beta))/4}\big\}
\nonumber\\
&=o\Big\{\big(n^{-b_{min}(\beta)/2-\delta/2}\big)^{(\beta-2)/\beta}+
\pmb{1}\{g_2(\beta)\le 2/3\}n^{-(1+3 g_2(\beta))/6}\nonumber\\
& \qquad+\pmb{1}\{g_2(\beta)>2/3\}n^{-1/2-(b_{min}(\beta)+2\delta)(3 g_2(\beta)-2)/4}
\Big\}\nonumber\\
&=o\big(n^{-b_{min}(\beta)(\beta-2)/(2\beta)}\big).
\label{pf:thm5.4}
\end{align}
If $bn^{-2/3}\rightarrow\infty$, then
$\ell\propto n/b$ and $k_1\propto n^{-1+\delta'}b$, so that
\begin{align}
\Xi_n&=O\big\{(n^{-1+\delta'}b)^{(\beta-2)/\beta}+
n^{3(1-\delta')(1-g_2(\beta))/2}b^{-(5-3 g_2(\beta))/2}\big\}\nonumber\\
&=o\big(n^{\{b_{max}(\beta)-1-\delta/2\}(\beta-2)/\beta}+
n^{-(1+3 g_2(\beta))/6}\big).
\label{pf:thm5.5}
\end{align}
That (\ref{pf:thm5.0}) holds under both cases (i) and (ii) follows from
(\ref{pf:thm5.1}) to (\ref{pf:thm5.5}).
\end{proof}

\begin{proof}[Proof of Theorem~\ref{thm:boot1.expo}]
Note that Theorem~\ref{thm:boot.expo}, (\ref{bias.cond}) and the orders prescribed 
of $(b,\ell,k_1)$ together imply 
\begin{eqnarray*}
\lefteqn{
\prob\big(\hat{T}^*_{b,\ell,k_1}\le y\big|X_1,\ldots,X_n\big)-
\Phi\left(\dfrac{y-n^{1/2}h^{5/2}f''(x_0)\mu_2/2}{\sqrt{f(x_0)\nu_2}}\right)}\\
&=&
O_p\big\{k_1\log\ell+n^{-1/2}\ell^{1/2}+(b\ell k_1)^{-1/2}\big\}=O_p\big\{n^{-1/3}(\log n)^{1/3}L_n^{-1}\big\},
\end{eqnarray*}
which proves (\ref{thm:boot1.expo}).
\end{proof}

\begin{proof}[Proof of Theorem~\ref{thm:boot2}]

Write for brevity 
\begin{align*}
\Psi_n &= n^{-1/2}(b\ell)^{1/2}+\big(nh^5/(b\ell)\big)^{(\beta-2)/(5\beta)}\\
& \quad +n^\delta\big(nh^5)^{(6g_1(\beta)-3)/10}(b\ell)^{-(1+3 g_1(\beta))/5}.
\end{align*}
The assumptions on $(b,\ell)$ require that
\begin{equation}
\label{pf:thm7.1}
\begin{gathered}
 (nh^5)^{-1/4}\ell^{5/2}=O(b\ell),\quad b\ell=O\big(
\min\big\{n,n^{1-\epsilon}h^5\ell^5\big\}\big),\\
(nh^5)^{-1/2}n^{2\epsilon/5}=O(\ell).
\end{gathered}
\end{equation}
Note that under (\ref{pf:thm7.1}), $b\ell$ has a smallest possible order $(nh^5)^{-3/2}n^\epsilon$ and
a largest possible order $n$.

If $h$ satisfies the condition given in case (i), then $\Psi_n=O\big\{n^{-1/2}(b\ell)^{1/2}\big\}$, which is minimised by
setting $b\ell$ and $\ell$ to have the smallest possible orders permitted under (\ref{pf:thm7.1}), that is
$b\ell\propto(nh^5)^{-1/4}\ell^{5/2}$ and $\ell\propto(nh^5)^{-1/2}n^{2\epsilon/5}$. This leads to the choice 
of $(b,\ell)$ specified in (i), under which $\Psi_n=O\big(n^{-5/4+\epsilon/2}h^{-15/4}\big)$.

If $h$ satisfies the condition given in case (ii), then 
\[\Psi_n=O\Big\{
\big(nh^5/(b\ell)\big)^{(\beta-2)/(5\beta)}
+n^\delta\big(nh^5)^{(6g_1(\beta)-3)/10}(b\ell)^{-(1+3 g_1(\beta))/5}\Big\} \]
whenever $b\ell$ has an order between $(nh^5)^{-3/2}n^\epsilon$ and
$nh^{(10\beta-20)/(7\beta-4)}$, and $\Psi_n=O\big\{n^{-1/2}(b\ell)^{1/2}\big\}$
whenever $b\ell$ has an order between $nh^{(10\beta-20)/(7\beta-4)}$ and $n$.
It follows that $\Psi_n$ attains a minimum order $h^{(5\beta-10)/(7\beta-4)}$ at
$b\ell\propto nh^{(10\beta-20)/(7\beta-4)}$, substitution of which into (\ref{pf:thm7.1})
yields the permissible range for $\ell$.

If $h$ satisfies the condition given in case (iii), then 
\[\Psi_n=O\Big\{n^\delta\big(nh^5)^{(6g_1(\beta)-3)/10}(b\ell)^{-(1+3 g_1(\beta))/5}\Big\}\]
whenever $b\ell$ has an order between 
$\big(n^{(6 g_1(\beta)+2+10\delta)}h^{30 g_1(\beta)-15}\big)^{1/(6g_1(\beta)+7)}$ and $(nh^5)^{-3/2}n^\epsilon$,
and $\Psi_n=O\big\{n^{-1/2}(b\ell)^{1/2}\big\}$
whenever $b\ell$ increases beyond the latter range.
It follows that $\Psi_n$ has a minimum order $\big(n^{-1 + 2\delta}h^{6g_1(\beta)-3}\big)^{5/(12 g_1(\beta)+14)}$ at
\[b\ell\propto \big(n^{(6 g_1(\beta)+2+10\delta)}h^{30 g_1(\beta)-15}\big)^{1/(6g_1(\beta)+7)},\]
substitution of which into (\ref{pf:thm7.1})
yields the permissible range for $\ell$.
\end{proof}

\begin{proof}[Proof of Theorem~\ref{thm:boot2.expo}]
Let 
\[
\Psi_n'=n^{1/5}h(b\ell)^{-1/5}\log\ell+n^{-1/10}h^{-1/2}(b\ell)^{-2/5}
+n^{-1/2}(b\ell)^{1/2}.
\]
The assumptions on $(b,\ell)$ require that
\begin{equation}
\label{pf:thm8.1}
\ell=O(b\ell),\;\; b\ell=o\big(nh^5\ell^5\big),\;\;b\ell=O(n),\;\;
(nh^5)^{-1/4}=o(\ell).
\end{equation}
Note that under (\ref{pf:thm8.1}), $b\ell$ has a smallest possible order slightly
exceeding $(nh^5)^{-1/4}$ and
a largest possible order $n$.

Consider first the case $h=O\big\{n^{-7/25}(\log n)^{-18/25}\big\}$. Here
\[
\Psi'_n=O\big\{n^{-1/10}h^{-1/2}(b\ell)^{-2/5}\big\}
\]
if
$b\ell$ has an order larger than $(nh^5)^{-1/4}$ but not exceeding
$n^{4/9}h^{-5/9}$, whereas
\[
\Psi'_n=O\big\{n^{-1/2}(b\ell)^{1/2}\big\}
\]
if $b\ell$ has an order larger than $n^{4/9}h^{-5/9}$.
It follows that $\Psi'_n$ attains a minimum order $(nh)^{-5/18}$ at
$b\ell\propto n^{4/9}h^{-5/9}$, substitution of which into (\ref{pf:thm8.1})
yields the permissible range for $\ell$.

Consider next the case $n^{7/25}(\log n)^{18/25}h\rightarrow\infty$. Here
\[
\Psi'_n=O\big\{n^{1/5}h(b\ell)^{-1/5}\log\ell\big\}
\]
if
$b\ell$ has an order larger than $(nh^5)^{-1/4}$ but not exceeding
$n(h\log n)^{10/7}$, whereas
\[
\Psi'_n=O\big\{n^{-1/2}(b\ell)^{1/2}\big\}
\]
if $b\ell$ has an order larger than $n(h\log n)^{10/7}$.
It follows that $\Psi'_n$ attains a minimum order $(h\log n)^{5/7}$ at
$b\ell\propto n(h\log n)^{10/7}$, substitution of which into (\ref{pf:thm8.1})
yields the permissible range for $\ell$.
\end{proof}

\begin{proof}
[Proof of Theorem~\ref{thm:boot3}]

The order of departure from normality of the bootstrap distribution of $\check{T}^*_{b,\ell,k}$ follows directly from Theorem~\ref{thm:boot}.
Let $\psi_n=k^{(\beta-2)/\beta}+b^{-1/2}\ell^{-1/2+\delta}k^{-3/2+3g_1(\beta)}+n^{-1/2+\delta}b^{-1/2}\ell^{1/2}k^{-3(1-g_2(\beta))/2}$, which depends monotonically on $(b,\ell)$. To minimise $\psi_n$ we should set $b$ as large and $\ell$ as small as possible. Substituting $b\propto n/\ell$ and $\ell=k^{-1-\delta'}$, $\psi_n$ reduces to
$k^{(\beta-2)/\beta}+n^{-1/2}k^{-3/2+3g_1(\beta)-\delta(1+\delta')}+n^{-1+\delta}k^{-\{5-3g_2(\beta)\}/2-\delta'}$, with $k$ satisfying $kn^{1/3}\rightarrow\infty$.

For $\beta\in(2,4]$, $\psi_n$ is dominated by $k^{(\beta-2)/\beta}+n^{-1/2}k^{-3/2+3g_1(\beta)-\delta(1+\delta')}$, which is minimised by setting 
$k\propto \max\big\{
n^{-\beta/\{\beta (5-6 g_1(\beta))-4\}},n^{-1/3+\delta'}\big\}$. The minimal 
order of $\psi_n$ is then given by 
$O\big\{n^{-(\beta-2)/\{\beta (5-6 g_1(\beta))-4\}+\delta}\big\}$ for $\beta\in(2,\beta_1)$
and by $O\big\{n^{-(\beta - 2)/(3\beta)+\delta}\big\}$ for $\beta\in[\beta_1,4]$, provided that $\delta'$ is chosen sufficiently small.

For $\beta\in(4,\infty)$, $\psi_n$ is dominated by $k^{(\beta-2)/\beta}+n^{-1+\delta}k^{-\{5-3g_2(\beta)\}/2-\delta'}$, which is minimised by setting 
$k\propto \max\big\{
n^{-2 \beta/\{\beta(7-3 g_2(\beta))-4\}},n^{-1/3+\delta'}\big\}$.
For $\delta'$ sufficiently small, the above choice of $k$ yields a minimal order
for $\psi_n$, which is $O\big\{n^{-2 (\beta-2)/\{\beta(7-3 g_2(\beta))-4\}+\delta}\big\}$ for $\beta\in(4,\beta_2)$ and $O\big\{n^{-(\beta - 2)/(3\beta)+\delta}\big\}$ for $\beta\ge\beta_2$.
\end{proof}

\begin{proof}
[Proof of Theorem~\ref{thm:boot3.expo}]

The order of departure from normality of the bootstrap distribution of $\check{T}^*_{b,\ell,k}$ follows directly from Theorem~\ref{thm:boot.expo}.
The block bootstrap contribution to the above order, that is $k\log\ell+n^{-1/2}\ell^{1/2}+(b\ell k)^{-1/2}$, can be reduced by setting
$b$ as large and $\ell$ as small as possible, subject to the constraints
$b=O(n/\ell)$ and
$k\ell\rightarrow\infty$. This suggests the choices
$b\propto n/\ell$ and $\ell\propto (kL_n^3)^{-1}$, under which the error due to the bootstrap has the order $k|\log k|+(nkL_n^3)^{-1/2}$, which is minimised by setting
$k\propto n^{-1/3}(\log n)^{-2/3}L_n^{-1}$.
\end{proof}

\bibliographystyle{imsart-nameyear}
\bibliography{hybridboot-density}

\end{document}